\documentclass[a4paper, 11pt]{article}
\usepackage[T1]{fontenc}
\usepackage[utf8]{inputenc}
\usepackage{amsmath}
\usepackage{amsfonts}
\usepackage{amssymb}
\usepackage{amsthm}
\usepackage{graphicx}
\usepackage[english]{babel}
\usepackage{version}
\usepackage{nicefrac}
\usepackage{tipa} 
\usepackage{hyperref}

\theoremstyle{definition}
\newtheorem{defin}{Definition}
\newtheorem{ex}{Example}
\theoremstyle{plain}
\newtheorem{theo}{Theorem}[section]
\newtheorem{lemma}[theo]{Lemma}
\newtheorem{obs}[theo]{Remark}
\newtheorem{prop}[theo]{Proposition}
\newtheorem{cor}[theo]{Corollary}
\newtheorem{theorem}{Theorem}

\newtheorem*{theorem-no}{Theorem}
\newtheorem{question}{Question}

\newtheorem*{example-no}{Example}

\numberwithin{equation}{section}

\renewenvironment{abstract}
{\par\noindent\textbf{\abstractname.}\ \ignorespaces}
{\par\medskip}

\title{Ahlfors regular conformal dimension and Gromov-Hausdorff convergence}
\author{Nicola Cavallucci}
\date{}

\begin{document}
\maketitle
\begin{abstract}
	\footnotesize
	We prove that the Ahlfors regular conformal dimension is upper semicontinuous with respect to Gromov-Hausdorff convergence when restricted to the class of uniformly perfect, uniformly quasi-selfsimilar metric spaces. Moreover we show the continuity of the Ahlfors regular conformal dimension in case of limit sets of discrete, quasiconvex-cocompact group of isometries of uniformly bounded codiameter of $\delta$-hyperbolic metric spaces under equivariant pointed Gromov-Hausdorff convergence of the spaces.
\end{abstract}
\tableofcontents

\section{Introduction}
The Ahlfors regular conformal gauge of a metric space $(X,d)$ is the set $\mathcal{J}(X,d)$ of all metrics on $X$ that are quasisymmetric equivalent to $d$.
By definition a homeomorphism $F\colon (X,d_X) \to (Y,d_Y)$ is a quasisymmetric equivalence if there exists a strictly increasing map $\eta \colon [0,+\infty) \to [0,+\infty)$ with $\eta(0)=0$ such that
$$\frac{d_Y(F(x), F(x'))}{d_Y(F(x), F(x''))} \leq \eta\left( \frac{d_X(x,x')}{d_X(x,x'')} \right)$$
for every $x,x',x'' \in X$ with $d_X(x,x'')> 0$. The notion of quasisymmetric maps was introduced in \cite{TV80} and it has played an important role in the study of quasiconformal structure on metric spaces. The \emph{Ahlfors regular conformal dimension} of a metric space $(X,d)$ is defined as
$$\text{CD}(X,d) := \inf\lbrace \text{HD}(X,d') \text{ s.t. } d'\in \mathcal{J}(X,d)\rbrace,$$
where  $\text{HD}$ denotes the Hausdorff dimension.
In general $\mathcal{J}(X,d)$ can be empty, implying CD$(X,d)=+\infty$. On the other hand the conformal dimension of any doubling, uniformly perfect metric space is always finite by \cite[Corollary 14.5]{Hei01}. 
There is a special class of metric spaces that are doubling and uniformly perfect: the class of perfect quasi-selfsimilar metric spaces (Proposition \ref{prop-perfect-implies-uniform}).
\begin{defin}
	\label{defin-qss}
	Let $\rho_0>0$ and $L_0\geq 1$. A compact metric space $(X,d)$ is said $(L_0,\rho_0)$-quasi-selfsimilar (shortly $(L_0,\rho_0)$-q.s.s.) if for every open ball $B(x,\rho)$ in $X$ with $0<\rho \leq \rho_0$ there is a map $\Phi\colon \left(B(x,\rho), \frac{\rho_0}{\rho}\cdot d\right) \to X$ which is $L_0$-biLipschitz and such that $\Phi(B(x,\rho)) \supseteq B(\Phi(x),\frac{\rho_0}{L_0})$.
\end{defin} 
\noindent The notation $\frac{\rho_0}{\rho}\cdot d$ means the metric obtained by multiplying the original metric $d$ by the positive number $\frac{\rho_0}{\rho}$.  In other words a metric space is $(L_0,\rho_0)$-quasi-selfsimilar if every ball of radius smaller than $\rho_0$ is, up to rescaling it to the right size, biLipschitz comparable to a ball of radius exactly $\rho_0$ of the same space.
\vspace{2mm}

\noindent This notion arises naturally in the study of limit sets of Gromov-hyperbolic groups and semi-hyperbolic rational fractions, see for instance \cite{Sul82}, \cite{Hai07}, \cite{BK13} and \cite{BM17}. Examples of spaces that are quasi-selfsimilar include Lipschitz manifolds, simplicial complexes with a metric of fixed constant curvature on each simplex, self-similar fractals, boundaries of cocompact Gromov-hyperbolic spaces and Julia sets of semi-hyperbolic rational fractions. Of course if one puts natural geometric constraints to each of the above classes then it is possible to quantify the quasi-selfsimilarity constants in terms of the constraints.
\vspace{2mm}

\noindent The purpose of this paper is to study the behaviour of the Ahlfors regular conformal dimension on quasi-selfsimilar metric spaces under Gromov-Hausdorff convergence. 
\begin{theorem}
	\label{theorem-main-intro}
	Let $(X_n,d_n)$ be a sequence of compact, $a_0$-uniformly perfect, $(L_0,\rho_0)$-q.s.s. spaces. Suppose it converges in the Gromov-Hausdorff sense to $(X_\infty, d_\infty)$. Then
	$\textup{CD}(X_\infty, d_\infty)\geq \limsup_{n\to +\infty} \textup{CD}(X_n, d_n)$.
\end{theorem}
For quasi-selfsimilar spaces uniform perfectness is quantitatively equivalent to a uniform lower bound of the diameter of the balls $B(\Phi(x),\frac{\rho_0}{L_0})$ appearing in Definition \ref{defin-qss}, see Proposition \ref{prop-perfect-implies-uniform}.\\
Theorem \ref{theorem-main-intro} is false if the spaces are not $(L_0,\rho_0)$-q.s.s.: the sequence $X_n = [0,1/n] \subseteq \mathbb{R}$ converges in the Gromov-Hausdorff sense to $X_\infty = \lbrace 0 \rbrace$, but CD$(X_n, d_E)=1$ for every $n$ while CD$(X_\infty, d_E) = 0$. Here $d_E$ is the standard Euclidean metric.
Moreover the upper semicontinuity in Theorem \ref{theorem-main-intro} cannot be improved to continuity in general.
\begin{ex}
	\label{example-cantor-intro}
	Let $X_n$ be the set built in this way: we start with $[0,1]$ and we remove the central segment of length $\frac{1}{2n+1}$. We do the same for each of the two remaining parts. We continue this procedure infinitely many times and we call $X_n$ the resulting metric space endowed with the Euclidean metric $d_E$. For instance $X_1$ is the standard Cantor set. The sequence $X_n$ converges to $X_\infty = [0,1]$ in the Gromov-Hausdorff sense. However \textup{CD}$(X_n, d_E)=0$ for every $n$ by \cite[Proposition 15.11]{DS97} (see also \cite[Theorem 2.16]{Pia11}), while \textup{CD}$(X_\infty, d_E)=1$.
\end{ex}

On the other hand we have continuity in a particular setting. In \cite{Cav21ter} the author studied the class $\mathcal{M}(\delta,D)$ of triples $(X,x,\Gamma)$, where $X$ is a proper $\delta$-hyperbolic metric space, $\Gamma$ is a discrete, non-elementary, quasiconvex-cocompact, torsion-free group of isometries of $X$ with codiameter bounded above by $D$ and $x$ belongs to the quasiconvex hull of the limit set $\Lambda(\Gamma)$. We refer to Section \ref{sec-Gromov-application} for more details about these terms. One of the main results of \cite{Cav21ter} is the closeness of $\mathcal{M}(\delta,D)$ under equivariant pointed Gromov-Hausdorff convergence. Under this convergence it is possible to prove that the limit sets $\Lambda(\Gamma_n)$ and $\Lambda(\Gamma_\infty)$ are quasisymmetric equivalent for $n$ big enough (see Section \ref{sec-proof-B}).
As a consequence we get the following.
\begin{theorem}
	\label{theorem-main-introduction}
	Let $(X_n,x_n,\Gamma_n) \subseteq \mathcal{M}(\delta,D)$ be a sequence of triples converging in the equivariant pointed Gromov-Hausdorff sense to $(X_\infty,x_\infty,\Gamma_\infty)$. Then, for suitable metrics, the sequence $\Lambda(\Gamma_n)$ is uniformly perfect and uniformly q.s.s. and converges in the Gromov-Hausdorff sense to $\Lambda(\Gamma_\infty)$. Moreover $\textup{CD}(\Lambda(\Gamma_\infty)) = \lim_{n\to +\infty}\textup{CD}(\Lambda(\Gamma_n)).$
\end{theorem}
\noindent Here the Ahlfors regular conformal dimension is computed with respect to any visual metric on the limit sets, see Proposition \ref{prop-isometry-quasisymmetry} and the discussion below.
Motivated by Theorem \ref{theorem-main-introduction}, we propose the following question.
\begin{question}
	\label{question-continuity-general}
	Are there conditions on the metric spaces $X_n$ that ensures continuity of the conformal dimension under Gromov-Hausdorff convergence? 
\end{question}
\noindent It is useful to consider the following example (the author thanks M.Murugan for bringing the reference \cite{Tys00} to his attention).
\begin{ex}
	\label{conjecture}
	Let us repeat the construction of Example \ref{example-cantor-intro} in dimension $2$. For every $n$ we define $X_n$ in the following way: we start with $[0,1]^2$, we divide it into $(\frac{1}{2n+1})^2$ squares and we delete the central one. Now we do the same for every remaining squares. We repeat the procedure infinitely many times. We endow $X_n$ with the Euclidean metric $d_E$. For instance the space $X_1$ is the standard Sierpinski carpet. The sequence $(X_n, d_E)$ converges in the Gromov-Hausdorff sense to $X_\infty = [0,1]^2$, whose Ahlfors regular conformal dimension is $2$. In this case we have $\lim_{n\to +\infty}\textup{CD}(X_n, d_E)=2$ by a well-known argument (see for instance \cite[Theorem 3.4 and Example 3.2]{Tys00}).
\end{ex}
\noindent We will briefly discuss Question \ref{question-continuity-general} and the example above at the end of Section \ref{sec-upper}.

\section{Preliminaries}
\label{sec-basic}
We denote a metric space by $(X,d)$. The open (resp. closed) ball of center $x\in X$ and radius $\rho > 0$ is denoted by $B(x,\rho)$ (resp. $\overline{B}(x,\rho)$). \\
Given $r>0$ and $Y\subseteq X$ we say that a subset $S$ of $Y$ is $r$-separated if $d(x,y)> r$ for all $x,y\in S$, while a subset $N$ of $Y$ is a $r$-net if for all $y\in Y$ there exists $x\in N$ such that $d(x,y)\leq r$. It is straightforward from the definitions that a maximal $r$-separated subset of $Y$ is a $r$-net.\\
A metric space $(X,d)$ is said $D$-doubling if the cardinality of any $\frac{\rho}{2}$-separated subset inside any ball of radius $\rho$ is at most $D$.\\
A metric space $(X,d)$ is perfect if it has no isolated points, while it is said $a$-uniformly perfect, $0<a<1$, if $\overline{B}(x,\rho)\setminus B(x,a\cdot\rho) \neq \emptyset$ for all $x\in X$ and $0\leq \rho < \text{Diam}(X)$, where Diam denotes the diameter. Clearly every uniformly perfect metric space is perfect. \\
Gromov-Hausdorff convergence will be always considered in the class of compact metric spaces. With the notation $(X_n,d_n)\underset{\text{GH}}{\longrightarrow} (X_\infty, d_\infty)$ we mean that the sequence of compact metric spaces $(X_n,d_n)$ converges in the Gromov-Hausdorff sense to the compact metric space $(X_\infty, d_\infty)$.

\subsection{Ahlfors regular spaces}
A metric space $(X,d)$ is said $(A,s)$-Ahlfors regular, for given $A,s\geq 0$, if there is a measure $\mu$ on $X$ satisfying
$$\frac{1}{A}\cdot \rho^s \leq \mu(B(x,\rho)) \leq A\cdot \rho^s$$
for all $x\in X$ and all $0<\rho\leq \text{Diam}(X)$. 
The following lemma is classical.
\begin{lemma}
	\label{lemma-Ahlfors-perfect}
	Let $A_0\geq 1$ and $s_0>0$. Then there exists $0<a_0=a_0(A_0,s_0)<1$ such that
	every $(A,s)$-Ahlfors regular metric space $(X,d)$ with $A\leq A_0$ and $s\geq s_0$ is $a_0$-uniformly perfect.
\end{lemma}
\begin{proof}
	We claim that $(X,d)$ is $a$-uniformly perfect for all $a<A^{-2/s}$.
	Indeed for any such $a$, for every $x\in X$ and every $0<\rho\leq\text{Diam}(x)$ we have
	$$\mu(\overline{B}(x,\rho)) \geq \frac{1}{A}\rho^s, \qquad \mu(B(x,a\cdot\rho))\leq A\cdot a^s\rho^s < \frac{1}{A}\cdot\rho^s.$$
	Hence $\mu(\overline{B}(x,\rho) \setminus B(x,a\cdot\rho)) >0$ and in particular this set is not empty. It is then clear we can choose any $a_0 < A_0^{-2/s_0}$.
\end{proof}

\subsection{Quasi-selfsimilar spaces}

We collect now some basic properties of quasi-selfsimilar metric spaces, see Definition \ref{defin-qss}.
We say a quasi-selfimilar metric space $(X,d)$ has diameters bounded below by some $c_0>0$ if the ball $B(\Phi(x),\frac{\rho_0}{L_0})$ that appears in Definition \ref{defin-qss} has diameter $\geq c_0$ for every $x\in X$ and every $0<\rho\leq \rho_0$.

\begin{prop}[\textup{Compare with \cite[Lemma 2.2 and Proposition 2.3]{Pia11}}]
	\label{prop-perfect-implies-uniform}
	Let $(X,d)$ be a quasi-selfsimilar metric space as in Definition \ref{defin-qss}. Then:
	\begin{itemize}
		\item[(i)] it is doubling;
		\item[(ii)] if it is perfect then it is uniformly perfect;
		\item[(iii)] it is uniformly perfect if and only if it has diameters bounded below, quantitatively in terms of the relative constants and the diameter of $X$.
	\end{itemize}
\end{prop}
\begin{proof}
	If (i) is not true then for every $n\in \mathbb{N}$ there exist $x_n \in X$ and $\rho_{n} > 0$ such that there is a $\frac{\rho_{n}}{2}$-separated set inside $B(x_{n},\rho_{n})$ of cardinality $\geq n$. Up to pass to a subsequence we can suppose $\lim_{n\to +\infty} \rho_{n} = \rho_\infty \in [0,+\infty)$ and that $x_n$ converges to $x_\infty \in X$. If $\rho_\infty > 0$ then $X$ is not totally bounded, a contradiction. If $\rho_\infty = 0$ we can find $L_0$-biLipschitz maps $\Phi_n\colon \left(B(x_{n}, \rho_{n}), \frac{\rho_0}{\rho_{n}}\cdot d\right) \to X$ such that $\Phi_n(B(x_{n},\rho_{n})) \supseteq B(\Phi_n(x_{n}),\frac{\rho_0}{L_0})$ for every $n$ big enough. Hence there exists a $(\rho_0/2L_0)$-separated set inside $B(\Phi_n(x_{n}), \frac{\rho_0}{L_0})$ with cardinality $\geq n$. Once again this contradicts the compactness of $X$.

	\noindent Now we show (ii). Since $X$ is perfect and compact we have the following property, as in \cite[Lemma 2.2]{Pia11}: for all $\rho > 0$ there exists $d(\rho)>0$ such that Diam$(B(x,\rho)) \geq d(\rho)$ for every $x\in X$.
	Suppose now $X$ is not uniformly perfect: then for all $n\in\mathbb{N}$ there exist $x_n\in X$ and $0<\rho_n\leq \text{Diam}(X)$ such that $\overline{B}(x_n,\rho_n) \setminus B(x_n, \rho_n / n) = \emptyset$. Up to take a subsequence we can suppose that $x_n$ converges to $x_\infty$ and $\rho_n$ converges to $\rho_\infty$. Suppose first $\rho_\infty > 0$. Let $y$ be a point inside $B(x_\infty, \rho_\infty)$. It also belongs to $B(x_n,\rho_n)$ for $n$ big enough, and so it belongs to $B(x_n, \rho_n / n)$. In other words 
	$d(x_\infty, y) \leq d(x_\infty, x_n) + \frac{\rho_n}{n}$ for every $n$ big enough, i.e. $d(x_\infty, y) = 0$ and $x_\infty$ is an isolated point. This shows that $X$ is not perfect, a contradiction.\\
	Suppose now $\rho_\infty = 0$. For all $n$ big enough we take the map $\Phi_n \colon B(x_n,\rho_n) \to X$ given by Definition \ref{defin-qss}. Therefore
	$$B\left(\Phi_n(x_n), \frac{r_0}{L_0}\right) \subseteq \Phi_n(B(x_n,\rho_n)) = \Phi_n\left(B\left(x_n,\frac{\rho_n}{n}\right)\right).$$
	From one side we have Diam$\left(B\left(\Phi(x_n), \frac{r_0}{L_0}\right)\right) \geq d(\frac{r_0}{L_0}) > 0$ for every $n$. On the other hand 
	$$\text{Diam}\left(\Phi_n\left(B\left(x_n,\frac{\rho_n}{n}\right)\right)\right) \leq L_0\cdot \frac{\rho_0}{\rho_n} \cdot \frac{2\rho_n}{n} \underset{n \to + \infty}{\longrightarrow} 0,$$
	which is a contradiction.

	\noindent Finally we prove (iii). Let us suppose $X$ has diameters bounded below by $c_0>0$. We fix $x\in X$ and $0<\rho\leq \rho_0$. We take the map $\Phi\colon B(x,\rho) \to X$ given by the definition of quasi-selfsimilarity. Since $\Phi(B(x,\rho))$ contains a set with diameter $\geq c_0$, then there exists $y\in B(x,\rho)$ such that $d(\Phi(x), \Phi(y)) \geq \frac{c_0}{2}$. Therefore $d(x,y) \geq \frac{1}{L_0}\cdot \frac{\rho}{\rho_0} \cdot \frac{c_0}{2} = a(L_0,\rho_0,c_0) \cdot \rho$, with $0<a(L_0,\rho_0,c_0) =: a<1$. If $\rho$ is bigger than $\rho_0$ then we apply what said above to $\rho_0$ finding $B(x,\rho_0) \setminus B(x,a\cdot\rho_0) \neq \emptyset$, so $B(x,\rho) \setminus B(x,\frac{a\cdot \rho_0}{\rho}\cdot \rho) \neq \emptyset$. Since $X$ is compact we have $\frac{a\cdot \rho_0}{\rho} \geq \frac{a\cdot \rho_0}{\text{Diam}(X)} =: a_0 > 0$, showing that $X$ is $a_0$-uniformly perfect and $a_0$ depends only on $L_0,\rho_0,c_0$ and the diameter of $X$.\\
	Viceversa if $X$ is $a_0$-uniformly perfect then $\overline{B}(\Phi(x), \frac{\rho_0}{L_0}) \setminus B(\Phi(x), a_0\cdot\frac{\rho_0}{L_0}) \neq \emptyset$ for all $x\in X$ and $\Phi$ as in Definition \ref{defin-qss}. Therefore Diam$(B(\Phi(x), \frac{\rho_0}{L_0})) \geq a_0\cdot\frac{\rho_0}{L_0} =: c_0$.
\end{proof}
We remark that the diameter bounded below condition is part of the definition of quasi-selfsimilar spaces in \cite{Kle06} and in \cite{Pia11}. When an upper bound on the diameter of the metric space is fixed this condition is equivalent to bounded uniform perfectness of the metric space by Proposition \ref{prop-perfect-implies-uniform}. For instance in the context of Theorem \ref{theorem-main-intro} there is a uniform upper bound on the diameter of the spaces $X_n$, so the spaces $X_n$ have all diameter bounded below by some $c_0 > 0$ if and only if they are all $a_0$-uniformly perfect for some $0 < a_0 < 1$.


\section{Combinatorial modulus}
It is known that the conformal dimension of a metric space is closely related to the combinatorial modulus, see for instance \cite{BK13}, \cite{Pia12}, \cite{MT10} and the references therein. In this section we recall the definition of combinatorial modulus and we prove some technical lemma.
\vspace{1mm}

\noindent From now on we fix a $D$-doubling metric space $(X,d)$. For every $k\in \mathbb{N}$ we choose a $10^{-k}$-net $X_k$ of $X$.
To simplify the notation, given a real number $\lambda > 0$ and $k\in \mathbb{N}$, we will denote by $B_{\lambda,k}(x)$ the open ball of center $x$ and radius $\lambda \cdot 10^{-k}$, namely $B(x,\lambda\cdot 10^{-k})$. The same convention holds for closed balls.
\vspace{1mm}

\noindent A $(\lambda,k)$-path is a finite collection $\gamma = \lbrace q_j \rbrace_{j=0}^M$ of elements of $X_k$ satisfying $\overline{B}_{\lambda,k}(q_{j})\cap \overline{B}_{\lambda,k}(q_{j+1}) \neq \emptyset$ for all $j=0,\ldots,M-1$. The points $q_0$ and $q_M$ are called respectively the starting and the ending point of the path.
\vspace{1mm}

\noindent Given two subsets $E,F\subseteq X$ we denote by $P_{\lambda,k}(E,F)$ the set of $(\lambda,k)$-paths with starting point in $E$ and ending point in $F$.
We denote by $\mathcal{A}_{\lambda,k}(E,F)$ the set of admissible functions, i.e. functions $f\colon X_k \to [0,+\infty)$ such that $\sum_{i=0}^M f(q_i) \geq 1$ for every $\lbrace q_i\rbrace_{i=0}^M \in P_{\lambda,k}(E,F)$.
\vspace{1mm}

\noindent Given a real number $p\geq 0$ we define
$$p\text{-Mod}_{\lambda,k}(E,F) = \inf_{f\in \mathcal{A}_{\lambda,k}(E,F)}\sum_{q\in X_k} f(q)^p,$$
and we call it the $p$-modulus of the couple $(E,F)$ at level $(\lambda,k)$. The infimum is actually realized: any admissible function realizing the minimum is said optimal.
If there are no $(\lambda,k)$-paths joining $E$ and $F$ we set $p\text{-Mod}_{\lambda,k}(E,F) = 0$.

\begin{lemma}
	\label{lemma-base}
	If $E'\subseteq E$ and $F'\subseteq F$ then $p\textup{-Mod}_{\lambda,k}(E',F') \leq p\textup{-Mod}_{\lambda,k}(E,F)$.
\end{lemma}
\begin{proof}
	If $P_{\lambda,k}(E',F') = \emptyset$ then the result is trivial by definition. Otherwise every path in $P_{\lambda,k}(E',F')$ belongs to $P_{\lambda,k}(E,F)$. This implies that $\mathcal{A}_{\lambda,k}(E,F)\subseteq \mathcal{A}_{\lambda,k}(E',F')$, and the result follows from the definition.
\end{proof}

\noindent Let $1\leq L_1<L_2$ be two real numbers. For every $i\in \mathbb{N}$ and every point $y\in X_i$ we set
$$p\text{-Mod}_{\lambda,k}^{L_1,L_2}(y) := p\text{-Mod}_{\lambda,i+k}(\overline{B}_{L_1,i}(y),X\setminus B_{L_2,i}(y)).$$
We remark that this is a modulus at level $(\lambda,i+k)$. Finally we define
$$p\text{-Mod}_{\lambda,k}^{L_1,L_2}(X) = \sup_{i\in\mathbb{N}}\sup_{y\in X_i}p\text{-Mod}_{\lambda,k}^{L_1,L_2}(y).$$
We want to control how this quantity changes when $L_1,L_2$ and $\lambda$ change. We recall that $D$ denotes the doubling constant of $X$.
\begin{lemma}(cp. \cite{Pia11}, Lemma 4.4)
	\label{lemma-rescaling}
	Let $k,\lambda,p$ be fixed quantities as above.
	Let $1\leq L_1' \leq L_1 < L_2 \leq L_2'$. Then there exist $\ell\in\mathbb{N}$ and $C>0$, depending only on $L_1, L_1', L_2, L_2'$ and $D$, such that 
	$$p\textup{-Mod}_{\lambda,k}^{L_1',L_2'}(X) \leq p\textup{-Mod}_{\lambda,k}^{L_1,L_2}(X)$$
	and
	$$p\textup{-Mod}_{\lambda,k + \ell}^{L_1,L_2}(X) \leq C\cdot p\textup{-Mod}_{\lambda,k}^{L_1',L_2'}(X).$$
\end{lemma}
\begin{proof}
	For every $y\in X_i$, $i\in \mathbb{N}$, we have $\overline{B}_{L_1',i}(y) \subseteq \overline{B}_{L_1,i}(y)$ and $X\setminus B_{L_2',i}(y) \subseteq X\setminus B_{L_2,i}(y)$, so the first inequality follows by Lemma \ref{lemma-base}.\\
	In order to prove the second inequality we define $\ell$ as the minimum integer satisfying $10^{-\ell}\leq \frac{L_2 - L_1}{L_2' + L_1'}$. We fix $y\in X_i$ for some $i\in \mathbb{N}$ and we consider the set $X_{i+\ell}(y) = \lbrace z\in X_{i+\ell} \text{ s.t. } B_{L_1',i+\ell}(z) \cap \overline{B}_{L_1,i}(y) \neq \emptyset \rbrace$. 
	We fix any $(\lambda,i+\ell+k)$-path $\gamma = \lbrace q_j \rbrace_{j=0}^M$ joining $\overline{B}_{L_1,i}(y)$ and $X\setminus B_{L_2,i}(y)$. This means in particular that $d(y,q_0) \leq L_1\cdot 10^{-i}$ and $d(y,q_M)\geq L_2\cdot 10^{-i}$. We can find $z\in X_{i+\ell}$ such that $d(z,q_0)\leq 10^{-i-\ell}$, so by definition $z\in X_{i+\ell}(y)$. We claim that the $(\lambda,i+\ell+k)$-path $\gamma$ joins $\overline{B}_{L_1',i+\ell}(z)$ and $X\setminus B_{L_2',i+\ell}(z)$. Indeed we know that $d(z, q_0) \leq 10^{-i-\ell} \leq L_1' \cdot 10^{-i-\ell}$. Moreover $d(z,y)\leq L_1\cdot 10^{-i} + L_1'\cdot 10^{-i-\ell}$. Therefore
	$$d(z,q_N) \geq L_2\cdot 10^{-i} - L_1\cdot 10^{-i} - L_1'\cdot 10^{-i-\ell} \geq L_2'\cdot 10^{-i-\ell}$$
	by the choice of $\ell$.
	This means that any path $\gamma\in P_{\lambda,i+\ell+k}(\overline{B}_{L_1,i}(y), X\setminus B_{L_2,i}(y))$ belongs to $P_{\lambda, i+\ell+k}(\overline{B}_{L_1',i+\ell}(z), X\setminus B_{L_2',i+\ell}(z))$
	for some $z\in X_{i+\ell}(y)$.\\
	For each $z\in X_{i+\ell}(y)$ we take optimal functions $f_z \in \mathcal{A}_{\lambda,i+\ell+k}(\overline{B}_{L_1',i+\ell}(z), X\setminus B_{L_2',i+\ell}(z))$ and we define the function $f\colon X_{i+\ell+k} \to [0,+\infty)$ as 
	$$f(q)=\max_{z\in  X_{i+\ell}(y)}f_z(q).$$
	We claim $f\in \mathcal{A}_{\lambda, i+\ell+k}(\overline{B}_{L_1,i}(y), X\setminus B_{L_2,i}(y)).$
	Indeed every path $\lbrace q_j \rbrace_{j=0}^M \in P_{\lambda, i+\ell+k}(\overline{B}_{L_1,i}(y), X\setminus B_{L_2,i}(y))$ belongs to $P_{\lambda, i+\ell+k}(B_{L_1',i+\ell}(z), X\setminus B_{L_2',i+\ell}(z))$ for some $z\in X_{i+\ell}(y)$, therefore 
	$$\sum_{j=0}^Mf(q_j) \geq \sum_{j=0}^Mf_z(q_j) \geq 1.$$
	Finally we have:
	\begin{equation*}
		\begin{aligned}
			\sum_{q\in X_{i+\ell+k}}f(q)^p &= \sum_{q\in X_{i+\ell+k}}\max_{z\in  X_{i+\ell}(y)}f_z(q)^p \leq \sum_{z\in  X_{i+\ell}(y)} \sum_{q\in X_{i+\ell+k}}f_z(q)^p \\
			&= \sum_{z\in  X_{i+\ell}(y)} p\text{-Mod}^{L_1',L_2'}_{\lambda,k}(z) \leq C\cdot p\text{-Mod}^{L_1',L_2'}_{\lambda,k}(X),
		\end{aligned}
	\end{equation*}
	where $C$ is a constant depending only on the doubling constant $D$, on $\ell$ and on $L_1'$. This shows that
	$$p\text{-Mod}^{L_1,L_2}_{\lambda, k+\ell}(y) \leq C\cdot p\text{-Mod}^{L_1',L_2'}_{\lambda,k}(X).$$
	Since this is true for every $y\in X_i$ and for every $i$ we get 
	$$p\text{-Mod}^{L_1,L_2}_{\lambda,k+\ell}(X) \leq C\cdot p\text{-Mod}^{L_1',L_2'}_{\lambda,k}(X).$$
\end{proof}

\begin{lemma}
	\label{lemma-lambda-rescaling}
	Let $k\in\mathbb{N}$, $p\geq 0$, $1\leq L_1 < L_2$ and $2<\lambda \leq \lambda'$. Then there exist $\ell\in \mathbb{N}$ and $C>0$ depending only on $\lambda,\lambda',D$ such that
	$$p\textup{-Mod}^{L_1,L_2}_{\lambda,k}(X) \leq p\textup{-Mod}^{L_1,L_2}_{\lambda',k}(X)$$
	and
	$$p\textup{-Mod}^{L_1,L_2}_{\lambda',k + \ell}(X) \leq C\cdot p\textup{-Mod}^{L_1,L_2}_{\lambda,k}(X)$$
	for all $k> k_0 = \log_{10}\left(\frac{2}{L_2-L_1}\right)$.
\end{lemma}
\begin{proof}
	For every $y\in X_i$, $i\in \mathbb{N}$, we have $$P_{\lambda,k}(\overline{B}_{L_1,i}(y), X\setminus B_{L_2,i}(y)) \subseteq P_{\lambda',k}(\overline{B}_{L_1,i}(y), X\setminus B_{L_2,i}(y)).$$
	Therefore arguing as in the proof of Lemma \ref{lemma-base} we get
	$$p\text{-Mod}_{\lambda,k}^{L_1,L_2}(y) \leq p\text{-Mod}_{\lambda',k}^{L_1,L_2}(y).$$
	Taking the supremum on $i\in \mathbb{N}$ and $y\in X_i$ we obtain the first inequality.\\
	In order to show the second inequality we define $\ell$ as the smallest integer such that $\lambda'\cdot 10^{-\ell} \leq \frac{\lambda}{2} -1$. It is well defined since $\lambda > 2$. We restrict the attention to the integers $k$ bigger than $k_0$, so that $10^{-k} < \frac{L_2-L_1}{2}$.\\
	We fix $y\in X_i$, $i\in \mathbb{N}$, and a $(\lambda', i+k + \ell)$-path $\gamma = \lbrace q_j \rbrace_{j=0}^M$ joining $\overline{B}_{L_1,i}(y)$ to $X\setminus B_{L_2,i}(y)$. For every $j=0,\ldots,M$ we take a point $\tilde{q}_j \in X_{i+k}$ such that $d(q_j,\tilde{q}_j) \leq 10^{-i-k}$. We claim $\tilde{\gamma} = \lbrace \tilde{q}_j\rbrace_{j=0}^M$ is a $(\lambda, i+k)$-path joining $\overline{B}_{L_1',i}(y)$ to $X\setminus B_{L_2',i}(y)$, where $L_1' = L_1 + 10^{-k}$ and $L_2' = L_2 - 10^{-k}$. Indeed we have:
	$$d(y,\tilde{q}_0) \leq d(y,q_0) + d(q_0,\tilde{q}_0) \leq L_1\cdot 10^{-i} + 10^{-i-k} = L_1'\cdot 10^{-i},$$
	$$d(y,\tilde{q}_M) \geq d(y,q_M) - d(q_M,\tilde{q}_M) \geq L_2\cdot 10^{-i} - 10^{-i-k} = L_2'\cdot 10^{-i}$$
	and
	\begin{equation*}
		\begin{aligned}
			d(\tilde{q}_j, \tilde{q}_{j+1}) &\leq d(\tilde{q}_j,q_j) + d(q_j,q_{j+1}) + d(q_{j+1}, \tilde{q}_{j+1}) \\
			&\leq 2\cdot 10^{-i-k} + 2\lambda'\cdot 10^{-i-k-\ell} \\
			&\leq 2\cdot 10^{-i-k} + 2\left(\frac{\lambda}{2} -1\right)\cdot 10^{-i-k} \\
			& = \lambda \cdot 10^{-i-k} 
		\end{aligned}
	\end{equation*}
	for every $j=0,\ldots,M-1$.	Observe that the condition on $k$ implies $L_1'<L_2'$.
	We are ready to compare the combinatorial moduli. We take an optimal function $\tilde{f} \in \mathcal{A}_{\lambda,i+k}(\overline{B}_{L_1',i}(y), X\setminus B_{L_2',i}(y))$ and we define the function $f\colon X_{i+k+\ell} \to [0,+\infty)$ by 
	$$f(q) := \max\lbrace \tilde{f}(\tilde{q}) \text{ s.t. } \tilde{q}\in X_{i+k} \text{ and } d(q,\tilde{q}) \leq 10^{-i-k} \rbrace.$$
	First of all we show that $f\in \mathcal{A}_{\lambda',i+k+\ell}(\overline{B}_{L_1,i}(y), X\setminus B_{L_2,i}(y))$. Indeed we have seen that given any $(\lambda',i+k+\ell)$-path $\lbrace q_j \rbrace_{j=0}^M$ joining $\overline{B}_{L_1,i}(y)$ to $X\setminus B_{L_2,i}(y)$ there is an associated $(\lambda, i+k)$-path $\lbrace \tilde{q}_j \rbrace_{j=0}^M$ joining $\overline{B}_{L_1',i}(y)$ to $X\setminus B_{L_2',i}(y)$ such that $d(q_j,\tilde{q}_j) \leq 10^{-i-k}$ for every $j=0,\ldots,M$. Therefore by definition of $f$ we have
	$$\sum_{j=0}^M f(q_j) \geq \sum_{j=0}^M\tilde{f}(\tilde{q}_j) \geq 1.$$
	Finally we observe that
	\begin{equation*}
		\begin{aligned}
			p\text{-Mod}^{L_1,L_2}_{\lambda',k+\ell}(y) &\leq \sum_{q\in X_{i+k+\ell}}f^p(q) \leq C' \cdot \sum_{\tilde{q}\in X_{i+k}}\tilde{f}^p(\tilde{q}) \\
			&= C'\cdot p\text{-Mod}^{L_1',L_2'}_{\lambda,k}(y) \leq C' \cdot p\text{-Mod}^{L_1',L_2'}_{\lambda,k}(X),
		\end{aligned}
	\end{equation*}
	where $C'$ is a constant depending only on $D$. Since this is true for every $y \in X_i$ and for every $i\in \mathbb{N}$ we conclude that 
	$$p\text{-Mod}^{L_1,L_2}_{\lambda',k+\ell}(X) \leq C' \cdot p\text{-Mod}^{L_1',L_2'}_{\lambda,k}(X).$$
	This inequality is true for all $k\geq k_0$. Choosing $L_1'' = L_1 + 10^{-k_0}$ and $L_2'' = L_2 - 10^{-k_0}$ one concludes, using the easy inequality in Lemma \ref{lemma-rescaling}, that
	$$p\text{-Mod}^{L_1,L_2}_{\lambda',k+\ell}(X) \leq C' \cdot p\text{-Mod}^{L_1'',L_2''}_{\lambda,k}(X)$$
	for all $k\geq k_0$.
	An application of the non-trivial inequality of Lemma \ref{lemma-rescaling} concludes the proof.
\end{proof}
In order to normalize the notation from now on we choose for technical reasons $\lambda = 10$, $L_1=3$, $L_2 = 4$ and we set
$$p\text{-Mod}_k(y) := p\text{-Mod}_{10,i+k}(\overline{B}_{3,i}(y),X\setminus B_{4,i}(y))$$
for every $i\in \mathbb{N}$ and every point $y\in X_i$.
In the same way we put
$$p\text{-Mod}_k(X) = \sup_{i\in\mathbb{N}}\sup_{y\in X_i}p\text{-Mod}_k(y).$$

\section{Combinatorial modulus on quasi-selfsimilar spaces}
In this section we consider the class of quasi-selfsimilar metric spaces as defined in Definition \ref{defin-qss}. On these spaces the computation of the combinatorial moduli is easier. Before that we need an easy result.
\begin{lemma}
	\label{lemma-qss-propagation}
	If $X$ is $(L_0,\rho_0)$-q.s.s. then it is $(L_0,\rho_1)$-q.s.s. for every $0<\rho_1\leq \rho_0$.
\end{lemma}
\begin{proof}
	We fix $x\in X$ and $0<\rho\leq \rho_1$. We apply the definition of $(L_0,\rho_0)$-quasi-selfsimilarity to the ball $B(x,\frac{\rho_0}{\rho_1}\rho)$: we can find a $L_0$-biLipschitz map $\Phi \colon  \left(B(x,\frac{\rho_0}{\rho_1} \rho), \frac{\rho_1}{\rho}\cdot d\right) \to X$  such that $\Phi(B(x,\frac{\rho_0}{\rho_1}\rho)) \supseteq B(\Phi(x),\frac{\rho_0}{L_0})$.
Then it is straightforward to see that the restriction of $\Phi$ to $\left(B(x,\rho), \frac{\rho_1}{\rho}\cdot d\right)$ is still $L_0$-biLipschitz. We now take any point $z\in B(\Phi(x),\frac{\rho_1}{L_0})$. We know there exists a point $y\in B(x,\frac{\rho_0}{\rho_1}\rho)$ such that $z=\Phi(y)$. By the property of $\Phi$ we get $d(x,y)< \rho$. This shows that $\Phi(B(x,\rho)) \supseteq B(\Phi(x),\frac{\rho_1}{L_0})$.
\end{proof}

Let $X$ be a $(L_0,\rho_0)$-q.s.s. space. We denote by $i_0$ the smallest integer such that $2(L_0 + 5)^2\cdot 10^{-i_0} \leq \rho_0$. We define $I_0 = \lbrace i \in \mathbb{N} \text{ s.t. } (L_0 + 5)\cdot 10^{-i} \geq 10^{-i_0}\rbrace$. Observe that the set $I_0$ is of the form $\lbrace 1,\ldots,n_0\rbrace$ where $n_0$ depends only on $L_0$ and $\rho_0$. We set
$$p\text{-Mod}_k(X,n_0):= \sup_{i\leq n_0}\sup_{y\in X_{i}}p\text{-Mod}_k(y).$$ 
The following is the main result of the section: it allows to use \emph{the fixed sizes up to $n_0$} to estimate the combinatorial modulus. Since the explicit doubling constant of our metric space plays an important role we sometimes add it in the definition: we say a metric space is $(L_0,\rho_0, D_0)$-q.s.s. if it is $(L_0,\rho_0)$-q.s.s. and $D_0$-doubling. 
\begin{prop}
	\label{prop-key}
	Let $X$ be $(L_0,\rho_0, D_0)$-q.s.s. Then there exist a constant $C_0\geq 1$ and an integer $\ell_0$, both depending only on $L_0$ and $D_0$, such that
	$$p\textup{-Mod}_{k+\ell_0}(X,n_0) \leq p\textup{-Mod}_{k+\ell_0}(X) \leq C_0\cdot p\textup{-Mod}_{k}(X,n_0)$$
	for every $k\in\mathbb{N}$.
\end{prop}
\begin{proof}
	The first inequality is trivial since we are doing a supremum among less elements.
	In order to show the second inequality we fix $i\in \mathbb{N}$ and $y\in X_i$. Clearly we can suppose $i > n_0$. By Lemma \ref{lemma-qss-propagation} we know that $X$ is also $(L_0,2(L_0 + 5)^2\cdot 10^{-i_0})$-q.s.s..  Since $2(L_0 + 5)^3\cdot 10^{-i} < 2(L_0 + 5)^2\cdot 10^{-i_0}$, then there is a $L_0$-biLipschitz map 
	$$\Phi \colon  \left(B(y,2(L_0+5)^3\cdot 10^{-i}), \frac{10^{-i_0}}{(L_0+5)\cdot 10^{-i}}\cdot d\right) \to X.$$
%
%
	We choose a point $x\in X_{i_0}$ such that $d(x,\Phi(y)) \leq 10^{-i_0}$.
	We consider any $(10,i+k)$-path $\lbrace q_j \rbrace_{j=0}^M$ joining $\overline{B}_{1,i}(y)$ to $X\setminus B_{(L_0 + 5)^3,i}(y)$. This means 
	\begin{itemize}
		\item[-] $q_0 \in \overline{B}(y,10^{-i})$ and $q_M \notin B(y,(L_0 + 5)^3\cdot 10^{-i})$;
		\item[-] $\overline{B}(q_j, 10\cdot 10^{-i-k}) \cap \overline{B}(q_{j+1}, 10\cdot 10^{-i-k}) \neq \emptyset$ for every $j = 0,\ldots,M-1$.
	\end{itemize}
	Suppose that $q_j\in B_{2(L_0 + 5)^3,i}(y)$ for every $j=0,\ldots,M$. Then we can choose a point $\tilde{q}_j \in X_{i_0 + k -1}$ such that $d(\tilde{q}_j, \Phi(q_j)) \leq 10^{-i_0 - k +1}$ for every $j=0,\ldots,M$. By the property of $\Phi$ we get $d(\Phi(y), \Phi(q_0)) \leq \frac{L_0}{L_0+5}\cdot 10^{-i_0}$ and $d(\Phi(y),\Phi(q_M)) \geq \frac{(L_0 + 5)^2}{L_0}\cdot 10^{-i_0}$. Therefore we have:
	$$d(x,\tilde{q}_0) \leq d(x,\Phi(y)) + d(\Phi(y),\Phi(q_0)) + d(\Phi(q_0), \tilde{q}_0) \leq 3\cdot 10^{-i_0},$$
	$$d(x,\tilde{q}_M) \geq d(\Phi(y),\Phi(q_M)) - d(x,\Phi(y)) - d(\Phi(q_M), \tilde{q}_M) \geq (L_0 +3)10^{-i_0} \geq 4\cdot 10^{-i_0}.$$
	Moreover we know that $d(q_j,q_{j+1})\leq 20\cdot 10^{-i-k}$ for every $j=0,\ldots,M-1$. Therefore we get
	\begin{equation*}
		\begin{aligned}
			d(\tilde{q}_j, \tilde{q}_{j+1}) &\leq d(\tilde{q}_j, \Phi(q_j)) + d(\Phi(q_j), \Phi(q_{j+1}) + d(\Phi(q_{j+1}), \tilde{q}_{j+1}) \\
			&\leq 10^{-i_0-k+1} + 20\cdot 10^{-i_0-k} + 10^{-i_0-k+1} \\
			&\leq 10 \cdot 10^{-i_0 - k + 1}.
		\end{aligned}
	\end{equation*}
	In other words $\lbrace \tilde{q}_j\rbrace_{j=0}^M$ is a $(10,i_0+k-1)$-path joining $\overline{B}_{3,i_0}(x)$ to $X\setminus B_{4,i_0}(x)$. \\
	We take an optimal map $\tilde{f}\in \mathcal{A}_{10, i_0+k-1}(\overline{B}_{3,i_0}(x), X\setminus B_{4,i_0}(x))$. We define the map $f\colon X_{i+k} \to [0,+\infty)$ by
	$$f(q) := \max\lbrace \tilde{f}(\tilde{q}) : \tilde{q}\in X_{i_0+k-1}\cap \overline{B}(\Phi(q), 10^{-i_0-k+1})\rbrace$$
	if $q\in \overline{B}(y,2(L_0 + 5)^3\cdot 10^{-i})$ and $0$ otherwise.
	We want to show that $f\in \mathcal{A}_{i+k}(\overline{B}_{1,i}(y), X\setminus B_{(L_0 + 5)^3,i}(y))$. We consider any path  $\lbrace q_j \rbrace_{j=0}^M \in P_{10,i+k}(\overline{B}_{1,i}(y), X\setminus B_{(L_0 + 5)^3,i}(y))$. First of all we can extract the minimal subpath $\lbrace q_j \rbrace_{j=0}^{M'}$ such that $q_{M'} \notin B_{(L_0 + 5)^3,i}(y)$. Clearly if $\sum_{j=0}^{M'} f(q_j) \geq 1$ then also $\sum_{j=0}^{M} f(q_j) \geq 1$, so it is enough to check the admissibility condition on this minimal subpath.
	For such a minimal subpath we can construct the path $\lbrace \tilde{q}_j \rbrace_{j=0}^{M'}$ as in the first part of the proof since  $q_j\in B_{2(L_0 + 5)^3,i}(y)$ for every $j=0,\ldots,M'$. By definition it holds:
	$$\sum_{j=0}^{M'} f(q_j) \geq \sum_{j=0}^{M'} \tilde{f}(\tilde{q}_j) \geq 1.$$
	Moreover we have
	\begin{equation*}
		\sum_{q\in X_{i+k}} f(q)^p \leq D_0 \cdot \sum_{\tilde{q}\in X_{i_0+k-1}} \tilde{f}(\tilde{q})^p \leq D_0\cdot p\text{-Mod}_{k-1}(x) \leq D_0\cdot p\text{-Mod}_{k-1}(X,n_0),
	\end{equation*}
	since $x\in X_{i_0}$ and $1\leq i_0\leq n_0$ by definition.
	By the arbitrariness of $i\in \mathbb{N}$ and $y \in X_i$ we conclude
	$$p\text{-Mod}_{10,k}^{1,(L_0 + 5)^3}(X) \leq D_0\cdot p\text{-Mod}_{k-1}(X,n_0)$$
	for every $k\geq 1$. Using Lemma \ref{lemma-rescaling} we obtain the second inequality, indeed:
	\begin{equation*}
		\begin{aligned}
			p\text{-Mod}_k(X) = p\text{-Mod}_{10,k}^{3,4}(X) &\leq C \cdot p\text{-Mod}_{10,k - \ell}^{1, (L_0+5)^3}(X) \\
			&\leq C\cdot D_0 \cdot p\text{-Mod}_{k- \ell -1}(X,n_0),
		\end{aligned}
	\end{equation*}
	where $C$ and $\ell$ are constants depending only on $L_0$ and $D_0$. The thesis follows choosing $C_0 = C\cdot D_0$ and $\ell_0 = \ell + 1$.	

\end{proof}

The Ahlfors regular conformal dimension of a compact, doubling, uniformly perfect metric space $(X,d)$ coincides with the critical exponent of the combinatorial modulus.
\begin{theo}[\textup{\cite[Theorem 4.5]{Pia11}}]
	\label{theo-CD-modulus}
	Let $(X,d)$ be a compact, doubling, uniformly perfect metric space. Then
		$$\textup{CD}(X,d) = \inf\lbrace p\geq 0 \textup{ s.t. } \liminf_{k\to+\infty} p\textup{-Mod}_k(X) = 0 \rbrace.$$
\end{theo}

By Lemma \ref{lemma-rescaling} and Lemma \ref{lemma-lambda-rescaling}, the right hand quantity does not depend on our specific choices of $\lambda=10$, $L_1=3$ and $L_2 = 4$ in the definition of $p\text{-Mod}_k(X)$: the critical exponent associated to any other admissible choice of $\lambda,L_1,L_2$ equals the Ahlfors regular conformal dimension of $(X,d)$.
Moreover, following again \cite{Pia11} and \cite{BK13}, in the quasi-selfsimilar setting it is possible to find a uniform estimate which will be the key ingredient of the proof of Theorem \ref{theorem-main-intro}.
\begin{prop}
	\label{prop-uniform-lower-bound-p-modulus}
	Let $(X,d)$ be a perfect $(L_0,\rho_0, D_0)$-q.s.s. metric space and let $p< \textup{CD}(X,d)$. Then there exists a constant $\lambda_0$ depending only on $D_0$, $L_0$ and $p$ such that
	$$p\textup{-Mod}_k(X,n_0) \geq \lambda_0 > 0$$
	for all $k> 0$.
\end{prop}
\begin{proof}
	$(X,d)$ is uniformly perfect and doubling by Proposition \ref{prop-perfect-implies-uniform}, so the Ahlfors regular conformal dimension of $(X,d)$ can be computed as in Theorem \ref{theo-CD-modulus}. The result follows by a submultiplicative estimate. \cite[Lemma 4.9]{Pia11} proves
	$$p\text{-Mod}_{10,k+h}^{1,4}(X) \leq C\cdot p\text{-Mod}_{10,k}^{\frac{11}{10},\frac{39}{10}}(X) \cdot p\text{-Mod}_{10,h}^{1,4}(X)$$
	for all $k,h\geq 0$. Here $C$ is a constant depending only on $p$ and $D_0$. Applying Lemma \ref{lemma-rescaling} we get
	$$p\text{-Mod}_{k+h}(X) \leq C'\cdot p\text{-Mod}_{k-\ell}(X) \cdot p\text{-Mod}_{h}(X)$$
	for all $k\geq \ell$ and $h\geq 0$, where $C'$ is a constant depending only on $p$ and $D_0$ and $\ell$ is a universal constant. Let us denote by $a_k$ the quantity $p\text{-Mod}_{k}(X)$. The inequality above is $a_{k+h}\leq C'\cdot a_{k-\ell}\cdot a_h$. By Theorem \ref{theo-CD-modulus} $\liminf_{k\to+\infty}a_k > 0$ since $p< \text{CD}(X,d)$. This implies that $a_k \geq \frac{1}{C'}$ for all $k > 0$. Indeed, if there exists $k > 0$ such that $C'\cdot a_k < (1-\varepsilon)$ for some $\varepsilon > 0$ then 
	$$a_{n(k+\ell)} \leq C'\cdot a_{k}\cdot a_{(n-1)(k+\ell)} \leq \cdots \leq (1-\varepsilon)^n$$
	for all $n\in\mathbb{N}$. Therefore the subsequence $\lbrace a_{n(k+\ell)}\rbrace_{n\in\mathbb{N}}$ would converge to $0$, which is a contradiction. Hence we have found a constant $\lambda > 0$ depending only on $p$ and $D_0$ such that $a_k\geq \lambda$ for all $k > 0$.
	An application of Proposition \ref{prop-key} gives the thesis.
\end{proof}

\section{Upper semicontinuity of the conformal dimension}
\label{sec-upper}
Our scope is to study the behaviour of the Ahlfors regular conformal dimension under Gromov-Hausdorff convergence. For technical reasons it is often useful to study ultralimits instead of Gromov-Hausdorff limits. It essentially avoids to extract converging subsequences. For more detailed notions on ultralimits we refer to \cite{DK18} and \cite{CavS20}. A non-principal ultrafilter $\omega$ is a finitely additive measure on $\mathbb{N}$ such that $\omega(A) \in \lbrace 0,1 \rbrace$ for every $A\subseteq \mathbb{N}$ and $\omega(A)=0$ for every finite subset of $\mathbb{N}$. Accordingly we write $\omega$-a.s. and for $\omega$-a.e.$(n)$ in the usual measure theoretic sense. \\
Given a bounded sequence $(a_n)$ of real numbers and a non-principal ultrafilter $\omega$ there exists a unique $a\in \mathbb{R}$ such that for every $\varepsilon > 0$ the set $\lbrace n \in \mathbb{N} \text{ s.t. } \vert a_n - a \vert < \varepsilon\rbrace$ has $\omega$-measure $1$, see for instance \cite[Lemma 10.25]{DK18}. The real number $a$ is called the ultralimit of the sequence $a_n$ and it is denoted by $\omega$-$\lim a_n$.\\
If $(X_n, d_n, x_n)$ is a sequence of pointed metric spaces, we denote by $(X_\omega, d_\omega, x_\omega)$ the ultralimit pointed metric space. It is the set of sequences $(y_n)$, where $y_n\in X_n$ for every $n$, such that $\omega$-$\lim d(x_n,y_n) < + \infty$ for $\omega$-a.e.$(n)$, modulo the relation $(y_n)\sim (y_n')$ if and only if $\omega$-$\lim d(y_n,y_n') = 0$. The point of $X_\omega$ defined by the class of the sequence $(y_n)$ is denoted by $y_\omega = \omega$-$\lim y_n$.
The formula $d_\omega(\omega$-$\lim y_n, \omega$-$\lim y_n') = \omega$-$\lim d(y_n,y_n')$ defines a metric on $X_\omega$ which is called the ultralimit distance on $X_\omega$.\\
The relation between Gromov-Hausdorff convergence and ultralimits is summarized here.
\begin{prop}[\cite{Jan17}, Proposition 3.11; \cite{Cav21ter}, Proposition 3.13]
	\label{prop-GH-ultralimit}
	Let $(X_n, d_n, x_n)$ be a sequence of pointed, compact metric spaces and let $\omega$ be a non-principal ultrafilter.
	\begin{itemize}
		\item[(i)] If $(X_n, d_n) \underset{\textup{GH}}{\longrightarrow} (X_\infty, d_\infty)$ then $(X_\omega, d_\omega)$ is isometric to $(X_\infty, d_\infty)$. In particular the ultralimit does not depend on the choice of the basepoints.
		\item[(ii)] If $(X_\omega, d_\omega, x_\omega)$ is compact then $(X_{n_k}, d_{n_k}) \underset{\textup{GH}}{\longrightarrow} (X_\omega, d_\omega)$ for some subsequence $\lbrace{n_k}\rbrace$.
	\end{itemize}
\end{prop}

Let $(X_n, d_{X_n}, x_n), (Y_n, d_{Y_n}, y_n)$ be two sequences of pointed metric spaces and let $\omega$ be a non-principal ultrafilter. A sequence of maps $f_n\colon X_n \to Y_n$ is said admissible if $\omega$-$\lim d_{Y_n}(f_n(x_n), y_n) < + \infty$. A sequence of admissible $L$-Lipschitz maps $f_n$ defines a $L$-Lipschitz map $f_\omega = \omega$-$\lim f_n \colon (X_\omega, x_\omega) \to (Y_\omega, y_\omega)$ by $f_\omega(\omega$-$\lim x_n) = \omega$-$\lim f_n(x_n)$.\\
The class of uniformly perfect $(L_0,\rho_0)$-quasi-selfsimilar metric spaces is closed under Gromov-Hausdorff convergence.
\begin{prop}
	\label{prop-closeness-GH}
	Let $(X_n,d_n)$ be a sequence of compact, $a_0$-uniformly perfect, $(L_0,\rho_0)$-q.s.s. metric spaces. Suppose it converges in the Gromov-Hausdorff sense to a metric space $(X_\infty, d_\infty)$. Then $(X_\infty, d_\infty)$ is a compact, $a_0$-uniformly perfect, $(L_0,\rho_0)$-q.s.s. metric space.
\end{prop}
\begin{proof}
	$X_\infty$ is compact by definition of Gromov-Hausdorff convergence. We fix a non-principal ultrafilter $\omega$ and we call $X_\omega$ the ultralimit space: it does not depend on the basepoints and it is isometric to $X_\infty$ by Proposition \ref{prop-GH-ultralimit}. We fix a point $x_\omega = \omega$-$\lim x_n \in X_\omega$ and a positive real number $\rho\leq \rho_0$. For every $n$ there exists a $L_0$-biLipschitz map $\Phi_n\colon \left(B(x_n,\rho), \frac{\rho_0}{\rho}\cdot d_n \right) \to X_n$ with $\Phi_n(B(x_n,\rho)) \supseteq B(\Phi_n(x_n), \frac{\rho_0}{L_0})$. The sequence of maps $\Phi_n$ is clearly admissible, so it defines a ultralimit $L_0$-biLipschitz map $\Phi_\omega$, which is defined on the ultralimit space of the sequence $\left(B(x_n,\rho), \frac{\rho_0}{\rho}\cdot d_n \right)$. We observe that this ultralimit space contains $B(x_\omega, \rho)$. Indeed if $y_\omega = \omega$-$\lim y_n \in B(x_\omega, \rho)$, then $d_n(y_n, x_n) < \rho$ $\omega$-a.s. Moreover the ultralimit metric of the metrics $\frac{\rho_0}{\rho}\cdot d_n$ is $\frac{\rho_0}{\rho}\cdot d_\omega$. So we can restrict $\Phi_\omega$ to a $L_0$-biLipschitz map from $\left(B(x_\omega,\rho), \frac{\rho_0}{\rho}\cdot d_\omega \right) \to X_\omega$. We need to show that $\Phi_\omega(B(x_\omega,\rho)) \supseteq B(\Phi_\omega(x_\omega), \frac{\rho_0}{L_0})$. We take $y_\omega = \omega$-$\lim y_n$ such that $d_\omega(y_\omega, \Phi_\omega(x_\omega)) \leq (1-2\varepsilon)\cdot \frac{\rho_0}{L_0}$, with $\varepsilon > 0$. By definition $d_n(y_n, \Phi_n(x_n)) \leq (1-\varepsilon)\cdot \frac{\rho_0}{L_0}$ for $\omega$-a.e.$(n)$. By assumption we can find points $z_n \in B(x_n,\rho)$ such that $\Phi_n(z_n) = y_n$, $\omega$-a.s. These points satisfy $$\frac{\rho_0}{\rho}\cdot d_n(z_n, x_n) \leq L_0\cdot d_n(y_n, \Phi_n(x_n)) \leq (1-\varepsilon)\cdot \rho_0,$$
	so $d_n(x_n,z_n) \leq (1-\varepsilon)\cdot \rho$. Clearly the point $z_\omega = \omega$-$\lim z_n$ belongs to $B(x_\omega, \rho)$ and satisfies $\Phi_\omega(z_\omega) = y_\omega$.\\
	It remains only to prove that $X_\omega$ is $a_0$-uniformly perfect. We fix $x_\omega = \omega$-$\lim x_n \in X_\omega$ and $0<\rho\leq \text{Diam}(X_\omega)$. For every $\varepsilon > 0$ we have $(1-\varepsilon)\rho \leq \text{Diam}(X_n)$ for $\omega$-a.e.$(n)$, therefore there exists a point $y_n^\varepsilon\in X_n$ with $d(x_n,y_n^\varepsilon)\leq (1-\varepsilon)\rho$ and $d(x_n,y_n^\varepsilon)\geq a_0(1-\varepsilon)\rho$ for $\omega$-a.e.$(n)$. We consider the ultralimit point $y_\omega^\varepsilon = \omega$-$\lim y_n^\varepsilon \in X_\omega$. It satisfies $d(x_\omega,y_\omega^\varepsilon)\leq (1-\varepsilon)\rho$ and $d(x_\omega,y_\omega^\varepsilon)\geq a_0(1-\varepsilon)\rho$. Since this is true for every $\varepsilon > 0$ and since $X_\omega$ is compact we can find a point $y_\omega \in X_\omega$ such that  $d(x_\omega,y_\omega)\leq \rho$ and $d(x_\omega,y_\omega)\geq a_0\rho$, showing that $X_\omega$ is $a_0$-uniformly perfect.
\end{proof}
\begin{obs}
	This proposition, together with Proposition \ref{prop-perfect-implies-uniform}, implies that the Gromov-Hausdorff limit of a sequence of compact $(L_0,\rho_0)$-q.s.s. metric spaces with diameters bounded below by $c_0 > 0$, as considered by \cite{Kle06} and \cite{Pia11}, is still uniformly perfect.
\end{obs}

We can now give the
\begin{proof}[Proof of Theorem \ref{theorem-main-intro}]
	We notice that since $X_\infty$ is compact then the diameters of $X_n$ are uniformly bounded above by some $\Delta_0 \geq 0$. We proceed in several steps. 	\\
%
	\noindent \textbf{Step 1:} \emph{there exists $D_0 \geq 0$ such that $X_n$ is $D_0$-doubling for every $n$.}\\
	Suppose it is not true: then for every $j\in \mathbb{N}$ there exist $n_j$, $x_{n_j} \in X_{n_j}$ and $\rho_{n_j} > 0$ such that there is a $\frac{\rho_{n_j}}{2}$-separated set inside $B(x_{n_j},\rho_{n_j})$ of cardinality $\geq j$. Up to pass to a subsequence we can suppose $\lim \rho_{n_j} = \rho_\infty \in [0,+\infty)$. Clearly $X_\infty$ is not totally bounded when $\rho_\infty > 0$, and this is impossible since $X_\infty$ is compact. If $\rho_\infty = 0$ we use the quasi-selfsimilarity to get $L_0$-biLipschitz maps $\Phi_j\colon \left(B(x_{n_j}, \rho_{n_j}), \frac{\rho_0}{\rho_{n_j}}\cdot d_{n_j}\right) \to X_{n_j}$ such that $\Phi_j(B(x_{n_j},\rho_{n_j})) \supseteq B(\Phi_j(x_{n_j}),\frac{\rho_0}{L_0})$ for every $j$ for which $\rho_{n_j} \leq \rho_0$. Hence we can find a $\frac{\rho_0}{2L_0}$-separated set inside $B(\Phi_j(x_{n_j}), \frac{\rho_0}{L_0})$ with cardinality $\geq j$. Once again this contradicts the compactness of $X_\infty$.
	\vspace{2mm}
	
	\noindent In order to simplify the notations we fix a non-principal ultrafilter $\omega$ and we call $X_\omega$ the ultralimit space, which is isometric to $X_\infty$ by Proposition \ref{prop-GH-ultralimit}.
	\vspace{1mm}
	
	\noindent \textbf{Step 2:} \emph{Let $k\in \mathbb{N}$. We fix a maximal $10^{-k}$-separated subset $X_{k,n}$ of $X_n$. Then:\vspace{-3mm}
		\begin{itemize}
			\item[(i)] the cardinality of $X_{k,n}$ is uniformly bounded from above, and each $X_{k,n}$ is a $10^{-k}$-net of $X_n$;
			\item[(ii)] the set $X_{k,\omega} := \lbrace \omega\textup{-}\lim q_n \textup{ s.t. } q_n \in X_{k,n}\rbrace$ is a $10^{-k}$-net of $X_\omega$;
			\item[(iii)] there exists $A_k\subseteq \mathbb{N}$, $\omega(A_k)=1$, such that the function $\pi_n \colon X_{k,\omega} \to X_{k,n}$, $\pi_n(\omega\textup{-}\lim q_n) := q_n$ is well defined and bijective for all $n\in A_k$.
	\end{itemize} }
	\noindent By Step 1 we know that each $X_n$ is $D_0$-doubling, therefore the cardinality of $X_{k,n}$ is uniformly bounded above in terms of $D_0$ and $k$. The second statement of (i) has been explained in Section \ref{sec-basic}.\\
	We take two points $\omega$-$\lim q_n$, $\omega$-$\lim q_n' \in X_{k,\omega}$. If $d_\omega(\omega\text{-}\lim q_n,\omega\text{-}\lim q_n') < 10^{-k}$ then $d_n(q_n,q_n')<10^{-k}$ $\omega$-a.s. and by definition $q_n = q_n'$ $\omega$-a.s., implying $\omega\text{-}\lim q_n = \omega\text{-}\lim q_n'$. Since $X_\omega$ is compact we conclude that the set $X_{k,\omega}$ is finite and that $\pi_n$ is well defined $\omega$-a.s.\\	
	We suppose $X_{k,\omega}$ is not a $10^{-k}$-net of $X_\omega$. Therefore we can find $y_\omega = \omega$-$\lim y_n \in X_\omega$ such that $d_\omega(y_\omega, q_\omega) > 10^{-k}$ for all $q_\omega \in X_{k,\omega}$. Since $X_{k,\omega}$ is finite we know that $d_n(y_n, q_n) > 10^{-k}$ for all $q_n\in X_{k,n}$, $\omega$-a.s. This contradicts the fact that $X_{k,n}$ is a $10^{-k}$-net for every $n$, so also $X_{k,\omega}$ is a $10^{-k}$-net of $X_\omega$.\\
	Since
	$$\vert d_\omega(q_\omega, q_\omega') - d_n(q_n, q_n')\vert < \frac{10^{-k}}{2}$$
	for all $q_\omega = \omega\text{-}\lim q_n, q_\omega' = \omega\text{-}\lim q_n' \in X_{k,\omega}$ and for $\omega$-a.e.$(n)$ we conclude that $\pi_n$ is injective $\omega$-a.s. Finally suppose $\pi_n$ is not surjective $\omega$-a.s. Then it is possible to find a set $A\in \omega$ such that for every $n\in A$ there exists $q_n\in X_{k,n}$ which is not in the image of $\pi_n$. In this case we consider the point $\omega$-$\lim q_n$ that belongs to $X_{k,\omega}$, finding a contradiction. This ends the proof of (iii).
	\vspace{2mm}
	
	\noindent Since we have fixed $10^{-k}$-nets $X_{k,n}$ and $X_{k,\omega}$ of $X_n$ and $X_\omega$ respectively, every path will be intended with respect to these sets.
	\vspace{1mm}
	
	\noindent \textbf{Step 3:} \emph{Let $k\in\mathbb{N}$. There exists a subset $B_k\subseteq \mathbb{N}$ of $\omega$-measure $1$ such that
	\begin{itemize}
		\item[(i)] the map $\pi_n\colon X_{k,\omega} \to X_{k,n}$ from Step 2 is well defined and bijective for every $n\in B_k$;
		\item[(ii)] for every $n\in B_k$, for every $(10,k)$-path $\gamma_n = \lbrace q_j^n\rbrace_{j=0}^M$ of $X_n$ the associated path $\gamma_\omega = \lbrace \pi_n^{-1}(q_j^n)\rbrace_{j=0}^M$ is a $(30,k)$-path of $X_\omega$.
	\end{itemize}}
	\noindent Since $X_{k,\omega}$ is finite we can find a subset $B_k$ of $A_k$ with $\omega$-measure $1$ such that 
	\begin{equation*}
		\vert d_\omega(q_\omega, q_\omega') - d_n(\pi_n(q_\omega), \pi_n(q_\omega')) \vert\leq 10\cdot 10^{-k}
	\end{equation*}
	for all $q_\omega,q_\omega' \in X_{k,\omega}$ and for all $n\in B_k$.
	Let us take a $(10,k)$-path $\gamma_n = \lbrace q_j^n\rbrace_{j=0}^M$ of $X_n$, $n\in B_k$. This means $d_n(q_j^n, q_{j+1}^n) \leq 20\cdot 10^{-k}$ for all $j=0,\ldots,M-1$.
	Therefore $d_\omega(\pi_n^{-1}(q_j^n), \pi_n^{-1}(q_{j+1}^n)) \leq 30\cdot 10^{-k}$, i.e. the thesis.
%
	\vspace{2mm}
	
	\noindent \textbf{Step 4:} \emph{Let $i,k\in\mathbb{N}$ and $p\geq 0$. Then $p\textup{-Mod}_{30,k}^{\frac{13}{4},\frac{15}{4}}(y_\omega) \geq \omega\textup{-}\lim p\textup{-Mod}_k(y_n)$ for every $y_\omega = \omega$-$\lim y_n \in X_{i,\omega}$.}
	\vspace{1mm}
	
	\noindent We apply Step 3 to the integer $i+k$ finding $B_{i+k}\subseteq \mathbb{N}$, $\omega(B_{i+k}) = 1$, and bijective maps $\pi_n \colon X_{i+k,\omega} \to X_{i+k,n}$ for all $n\in B_{i+k}$.
	We take an optimal function $f_\omega \in \mathcal{A}_{30,i+k}(\overline{B}_{13/4,i}(y_\omega), X_\omega \setminus B_{15/4,i}(y_\omega))$. By definition $f_\omega$ maps points of $X_{i+k,\omega}$ to $[0,+\infty)$. For all $n\in B_{i+k}$ we define the functions $f_n\colon X_{i+k,n} \to [0,+\infty)$ by $f_n(q) = f_\omega(\pi_n^{-1}(q))$. We find another subset $C_{i+k,y_\omega}\subseteq B_{i+k}$ of $\omega$-measure $1$ such that 
	\begin{equation*}
		\vert d_\omega(q_\omega, y_\omega) - d_n(\pi_n(q_n), y_n) \vert\leq \frac{1}{4}\cdot 10^{-i}
	\end{equation*}
	for all $q_\omega \in X_{i+k,\omega}$ and for all $n\in C_{i+k,y_\omega}$.
	We want to check that $f_n\in \mathcal{A}_{10,i+k}(\overline{B}_{3,i}(y_n), X_n \setminus B_{4,i}(y_n))$ for all $n\in C_{i+k,y_\omega}$. We fix $n\in C_{i+k,y_\omega}$ and we take a $(10,i+k)$-path $\gamma_n = \lbrace q_j^n\rbrace_{j=0}^M$ such that $d_n(y_n, q_0^n) \leq 3\cdot 10^{-i}$ and $d_n(y_n,q_M^n)> 4\cdot 10^{-i}$. We denote by $\gamma_\omega = \lbrace \pi_n^{-1}(q_j^n)\rbrace$ the $(30,i+k)$-path given by Step 3.	
	We observe that $d_\omega(y_\omega, \pi_n^{-1}(q_0^n)) \leq \frac{13}{4}\cdot 10^{-i}$ and $d_\omega(y_\omega,q_M^\omega)> \frac{15}{4}\cdot 10^{-i}$, i.e. the $(30,i+k)$-path $\gamma_\omega$ joins $\overline{B}_{13/4,i}(y_\omega)$ and $X_\omega \setminus B_{15/4,i}(y_\omega)$. By definition of $f_n$ we get
	$$\sum_{j=0}^{M} f_n(q_j^n) = \sum_{j=0}^M f_\omega(\pi_n^{-1}(q_j^n)) \geq 1.$$
	Moreover it holds
	$$p\text{-Mod}_k(y_n) \leq \sum_{q\in X_{i+k,n}}f_n^p(q) = \sum_{q\in X_{i+k,\omega}}f_\omega^p(\pi^{-1}_n(q)) = p\text{-Mod}_{30,k}^{\frac{13}{4},\frac{15}{4}}(y_\omega).$$
	Since this is true for all $n\in C_{i+k,y_\omega}$ we get 
	$$p\textup{-Mod}_{30,k}^{\frac{13}{4},\frac{15}{4}}(y_\omega) \geq \omega\textup{-}\lim p\textup{-Mod}_k(y_n).$$
	\vspace{1mm}
	
	\noindent \textbf{Step 5:} \emph{Conclusion.}\\
	We fix $k\in \mathbb{N}$ and $0\leq p < \omega$-$\lim \text{CD}(X_n,d_n)$. 
	By Proposition \ref{prop-uniform-lower-bound-p-modulus} we find a constant $\lambda_0 > 0$ depending only on $D_0$, $L_0$ and $p$ such that
	\begin{equation}
		\label{eq-uniform}
		\sup_{i\leq n_0}\sup_{y\in X_{i,n}}p\text{-Mod}_k(y) \geq \lambda_0
	\end{equation}
	for $\omega$-a.e.$(n)$. For all these $n$'s we take a point $y_n \in X_{i_n,n}$, $1\leq i_n \leq n_0$, realizing the supremum in \eqref{eq-uniform}. The sequence $i_n$ is $\omega$-a.s. equal to some $i_* \in \lbrace 1,\ldots,n_0\rbrace$. So the limit point $y_\omega$ belongs to $X_{i_*,\omega}$. By Step 4 we get
	$$p\textup{-Mod}_{30,k}^{\frac{13}{4},\frac{15}{4}}(X_\omega) \geq p\textup{-Mod}_{30,k}^{\frac{13}{4},\frac{15}{4}}(y_\omega) \geq \omega\textup{-}\lim p\textup{-Mod}_k(y_n) \geq \lambda_0.$$
	Applying the easy inequalities of Lemma \ref{lemma-rescaling} and Lemma \ref{lemma-lambda-rescaling} we conclude that $p\textup{-Mod}_k(X_\omega) \geq \lambda_0$	for every $k$. This shows $\liminf_{k\to+\infty}p\textup{-Mod}_k(X_\omega) > 0$.
	$X_\omega$ is a compact, doubling and uniformly perfect metric space by Proposition \ref{prop-closeness-GH} and Proposition \ref{prop-perfect-implies-uniform}, so we can apply Theorem \ref{theo-CD-modulus} to get $p\leq \text{CD}(X_\omega, d_\omega)$. In conclusion we proved 
	$$\omega\text{-}\lim \text{CD}(X_n,d_n) \leq \text{CD}(X_\omega,d_\omega).$$
	This inequality is true for every non-principal ultrafilter $\omega$, so by Lemma 6.3 of \cite{Cav21ter} we have 
	$\limsup_{n\to +\infty} \text{CD}(X_n,d_n) \leq \text{CD}(X_\infty,d_\infty).$
\end{proof}

The main tools used in this proof are: the reduction of the computation of the combinatorial modulus to a \emph{finite} set of scales and the uniform lower bound on the combinatorial modulus, independent of $k$, given by Proposition \ref{prop-uniform-lower-bound-p-modulus}. The approach to the lower semicontinuity problem is more difficult because from one side it is still possible to reduce the computation to a finite set of scales, but from the other side there is no more any control on the behaviour of the $p$-modulus, independent of $k$. If we take some $p > \text{CD}(X_n, d_n)$ for every $n$ then by Theorem \ref{theo-CD-modulus} it holds $\liminf_{k\to+\infty}p\text{-Mod}_k(X_n) = 0$. But a priori it is not possible to conclude that $\liminf_{k\to+\infty}p\text{-Mod}_k(X_\infty) = 0$. Indeed for given $\varepsilon > 0$ we cannot control the threshold $k_\varepsilon$ such that $p\text{-Mod}_k(X_n) < \varepsilon$ for $k\geq k_\varepsilon$. Clearly if we have this kind of uniform control on the spaces $X_n$ then the Ahlfors regular conformal dimension of $X_\infty$ is equal to the limit of the Ahlfors regular conformal dimensions of $X_n$. \\
Then Question \ref{question-continuity-general} can be rephrased in the following way: are there (interesting) geometric conditions on a quasi-selfsimilar space that gives a uniform control on the thresholds $k_\varepsilon$ defined above?\\
This question seems to be related to (uniform) weak super-multiplicative properties of the sequence $p$-Mod$_k(X)$, as studied in relation with the combinatorial Lowner property in \cite[§4, §8]{BK13}, and in the special case of the Sierpinski carpet in \cite[Theorem 1.3]{Kwa20}. This weak super-multiplicative property seems to hold true only for spaces in which curves are uniformly distributed in some sense, as suggested by the arguments used again in \cite[Lemma 4.3 and Lemma 8.1]{BK13}. 
This observation gives a possible approach to the question presented in the introduction in case of spaces satisfying a uniform combinatorial Lowner property.

\section{Gromov-hyperbolic spaces}
\label{sec-Gromov-application}
In this second part of the paper we prove Theorem \ref{theorem-main-introduction}. We briefly recall the definition of Gromov-hyperbolic metric spaces. Good references are for instance \cite{BH09} and \cite{CDP90}. Let $X$ be a metric space. Given  three points  $x,y,z \in X$,  the {\em Gromov product} of $y$ and $z$ with respect to $x$  is
\vspace{-3mm}

$$(y,z)_x = \frac{1}{2}\big( d(x,y) + d(x,z) - d(y,z) \big).$$

\noindent The space $X$ is said {\em $\delta$-hyperbolic}, $\delta \geq 0$, if   for every four points $x,y,z,w \in X$   the following {\em 4-points condition} hold:
\vspace{-2mm}
\begin{equation}\label{hyperbolicity}
	(x,z)_w \geq \min\lbrace (x,y)_w, (y,z)_w \rbrace -  \delta,
	\vspace{-2mm}
\end{equation}
or, equivalently,
\vspace{-2mm}
\begin{equation}
	\label{four-points-condition}
	d(x,y) + d(z,w) \leq \max \lbrace d(x,z) + d(y,w), d(x,w) + d(y,z) \rbrace + 2\delta. 
	\vspace{-2mm}
\end{equation}

\noindent The space $X$ is   {\em Gromov hyperbolic} if it is $\delta$-hyperbolic for some $\delta \geq 0$.

\noindent Let $X$ be a $\delta$-hyperbolic metric space and $x$ be a point of $X$. 
The {\em Gromov boundary} of $X$ is defined as the quotient 
\vspace{-2mm}
$$\partial X = \lbrace (z_n)_{n \in \mathbb{N}} \subseteq X \hspace{1mm} | \hspace{1mm}   \lim_{n,m \to +\infty} (z_n,z_m)_{x} = + \infty \rbrace \hspace{1mm} /_\approx, \vspace{-3mm}$$
where $(z_n)_{n \in \mathbb{N}}$ is a sequence of points in $X$ and $\approx$ is the equivalence relation defined by $(z_n)_{n \in \mathbb{N}} \approx (z_n')_{n \in \mathbb{N}}$ if and only if $\lim_{n,m \to +\infty} (z_n,z_m')_{x} = + \infty$.  \linebreak
We will write $ z = [(z_n)] \in \partial X$ for short, and we say that $(z_n)$ {\em converges} to $z$. This definition  does not depend on the basepoint $x$. There is a natural topology on $X\cup \partial X$ that extends the metric topology of $X$. 
The Gromov product can be extended to points $z,z'  \in \partial X$ by 
\vspace{-2mm}
$$(z,z')_{x} = \sup_{(z_n) , (z_n') } \liminf_{n,m \to + \infty} (z_n, z_m')_{x} \vspace{-3mm}$$
where the supremum is taken among all sequences such that $(z_n) \in z$ and $(z_n')\in z'$.
For every $z,z',z'' \in \partial X$ it continues to hold
\vspace{-2mm}
\begin{equation}
	\label{hyperbolicity-boundary}
	(z,z')_{x} \geq \min\lbrace (z,z'')_{x}, (z',z'')_{x} \rbrace - \delta.
	\vspace{-3mm}
\end{equation}
Moreover for all sequences $(z_n),(z_n')$ converging to  $z,z'$ respectively it holds
\vspace{-2mm}
\begin{equation}
	\label{product-boundary-property}
	(z,z')_{x} -\delta \leq \liminf_{n,m \to + \infty} (z_n,z_m')_{x} \leq (z,z')_{x}.
	\vspace{-3mm}
\end{equation}
The Gromov product between a point $y\in X$ and a point $z\in \partial X$ is defined in a similar way and it  satisfies a condition analogue of  \eqref{product-boundary-property}.\\
The boundary of a $\delta$-hyperbolic metric space is metrizable. A metric $D_{x,a}$ on $\partial X$ is called a {\em visual metric} of center $x \in X$ and parameter $a\in\left(0,\frac{1}{2\delta\cdot\log_2e}\right)$ if there exists $V> 0$ such that for all $z,z' \in \partial X$ it holds
\begin{equation}
	\label{visual-metric}
	\frac{1}{V}e^{-a(z,z')_{x}}\leq D_{x,a}(z,z')\leq V e^{-a(z,z')_{x}}.
\end{equation}
A visual metric is said {\em standard} if for all $z,z'\in \partial X$ it holds
\begin{equation}
	\label{eq-new-visual}
	(3-2e^{a\delta})e^{-a(z,z')_{x}}\leq D_{x,a}(z,z')\leq e^{-a(z,z')_{x}}.
\end{equation}
For all $a$ as before and $x\in X$ there exists always a standard visual metric of center $x$ and parameter $a$ (cp. \cite{Pau96}, \cite{BS11}). Any two different visual metrics are quasisymmetric equivalent, and the quasisymmetric homeomorphism is the identity (\cite[Lemma 6.1]{BS11}). This defines a well defined quasisymmetric gauge on $\partial X$, that we denote by $\mathcal{J}(\partial X)$. If $C$ is a subset of $\partial X$ then the restriction of two visual metrics on $C$ define again two quasisymmetric distances, so it is well defined the quasisymmetric gauge $\mathcal{J}(C)$.
\vspace{1mm}

\noindent We will deal with proper metric spaces, i.e. spaces in which every closed ball is compact. A metric space $X$ is $K$-almost geodesic if for all $x,y\in X$, for all $t\in [0,d(x,y)]$ there exists $z\in X$ such that $\vert d(x,z) - t\vert \leq K$ and $\vert d(y,z) - (d(x,y) - t) \vert \leq K$. If we do not need to specify the value of $K$ we simply say that $X$ is almost geodesic. A metric space is geodesic if it is $0$-almost geodesic. 
Let $X$ be a proper, geodesic, Gromov-hyperbolic metric space.
Every geodesic ray $\xi$ defines a point  $\xi^+=[(\xi(n))_{n \in \mathbb{N}}]$  of the Gromov boundary $ \partial X$. Moreover for every $z\in \partial X$ and every $x\in X$ it is possible to find a geodesic ray $\xi_{x,z}$ such that $\xi(0)=x$ and $\xi^+ = z$. Analogously, given different points $z = [(z_n)], z' = [(z'_n)] \in \partial X$ there exists  a geodesic line $\gamma$ joining $z$ to $z'$, i.e. such that  $\gamma|_{[0, +\infty)}$ and $\gamma|_{(-\infty,0]}$ join   $\gamma(0)$ to $z,z'$ respectively. We call $z$ and $z'$ the  {\em positive} and {\em negative endpoints} of $\gamma$, respectively,  denoted  $\gamma^\pm$. \\
\noindent The {\em quasiconvex hull} of a subset $C$ of $\partial X$ is the union of all the geodesic lines joining two points of $C$ and it is denoted by QC-Hull$(C)$. If $X$ is proper and geodesic and $C$ has more than one point then QC-Hull$(C)$ is non-empty by the discussion above. We can say more.

\begin{prop}
	\label{prop-convex-hull-boundary}
	Let $X$ be a proper, geodesic, $\delta$-hyperbolic metric space and let $C\subseteq \partial X$ be a closed subset with at least two points. Then $\textup{QC-Hull}(C)$ is proper, $36\delta$-almost geodesic and $\delta$-hyperbolic. Moreover $\mathcal{J}(\partial \textup{QC-Hull}(C)) \cong  \mathcal{J}(C)$, in the sense that there exists a homeomorphism $F\colon C \to \partial \textup{QC-Hull}(C)$ which is a quasisymmetric equivalence when we equip $C$ and $\partial \textup{QC-Hull}(C)$ with every metrics in the gauges $\mathcal{J}(C)$ and $\mathcal{J}(\partial \textup{QC-Hull}(C))$, respectively.
\end{prop}

We need the following approximation result.
\begin{lemma}[\cite{Cav21ter}, Lemma 4.6]
	\label{approximation-ray-line}
	Let $X$ be a proper, geodesic, $\delta$-hyperbolic metric space. Let $C\subseteq \partial X$ be a subset with at least two points and $x\in \textup{QC-Hull}(C)$. Then $d(\xi_{x,z}, \textup{QC-Hull}(C)) \leq 14\delta$ for all $z\in C$.
\end{lemma}

\begin{proof}[Proof of Proposition \ref{prop-convex-hull-boundary}]
	QC-Hull$(C)$ is closed and $36\delta$-quasiconvex (\cite[Lemma 4.5]{Cav21ter}), i.e. every point of every geodesic segment joining every two points $y,y'$ of $\text{QC-Hull}(C)$ is at distance at most $36\delta$ from $\text{QC-Hull}(C)$. This implies that $\text{QC-Hull}(C)$ is $36\delta$-almost geodesic and proper. Condition \eqref{hyperbolicity} involves only the distance function, so $\text{QC-Hull}(C)$ is $\delta$-hyperbolic. We define a map $F\colon C \to \partial \textup{QC-Hull}(C)$ in the following way. We fix $x\in \textup{QC-Hull}(C)$. For every $z\in C$ we take a sequence $(z_n)\in \textup{QC-Hull}(C)$ such that $d(\xi_{x,z}(n), z_n) \leq 14\delta$, as provided by Lemma \ref{approximation-ray-line}. The sequence $(z_n)$ defines a point $\hat{z} \in \partial \textup{QC-Hull}(C)$, since $\lim_{n,m\to +\infty} (z_n,z_m)_x = +\infty$. We set $F(z) :=\hat{z}$. It is straightforward to check that $F$ is well defined, i.e. it does not depend on the choice of the sequence $(z_n)$. The Gromov products on $C$ and $\partial \textup{QC-Hull}(C)$ are comparable by \eqref{product-boundary-property}, namely
	\begin{equation}
		\label{eq-comparison-products}
		(z,z')_x - \delta \leq (F(z), F(z'))_x \leq (z,z')_x.
	\end{equation}
	Fix visual distances $D_{x,a}$ and $\hat{D}_{x,a}$ on $\partial X$ and $\partial \textup{QC-Hull}(C)$. By \eqref{eq-comparison-products} and \eqref{visual-metric} we have that $F$ is injective. If moreover it is surjective then it is a quasisymmetric homemorphism from $(C,D_{x,a})$ to $(\textup{QC-Hull}(C), \hat{D}_{x,a})$, which is the thesis. So fix a point $\hat{z} \in \partial\textup{QC-Hull}(C)$. By definition it is represented by a sequence $(z_n) \in \textup{QC-Hull}(C)$ such that $\lim_{n,m \to +\infty}(z_n,z_m)_x = +\infty$. Let $\gamma_n$ be a geodesic line of $X$ such that $\gamma_n^{\pm} \in C$ and $z_n \in \gamma_n$. We claim that, up to change the orientation of $\gamma_n$, it holds $\lim_{n\to +\infty}(\gamma_n^+, z_n)_x = +\infty$, so that $\gamma_n^+$ converges to $\hat{z} \in \partial X$ as $n$ goes to $+\infty$. Since $C$ is closed we deduce that $\hat{z} \in C$. Let us prove the claim is false, so both $(z_n, \gamma_n^\pm) \leq M$ for every $n$, for some $M$. By \cite[Lemma 3.2]{CavS20bis} applied to both the segments $[z_n,\gamma_n^\pm]$ we get $d(x, [z_n, \gamma_n^\pm]) \leq M + 4\delta$. Let us call $p_n^\pm$ points on the rays $[z_n, \gamma_n^\pm]$ realizing the distance from $x$. The 4-point condition \eqref{four-points-condition} gives
	$$d(x,z_n) + d(p_n^+, p_n^-) \leq \max\lbrace d(x,p_n^+) + d(z_n, p_n^-), d(x,p_n^-) + d(z_n, p_n^+)  \rbrace + 2\delta.$$
	Since $d(p_n^+, p_n^-) = d(z_n, p_n^-) + d(z_n, p_n^+)$ the inequality above implies
	$$d(x,z_n) \leq \max\lbrace d(x,p_n^+), d(x,p_n^-)  \rbrace + 2\delta \leq 2M + 10\delta.$$
	But this is impossible since $\lim_{n\to +\infty}d(x,z_n) = +\infty$.
\end{proof}

The quasisymmetric gauge of an almost geodesic Gromov-hyperbolic space is preserved by quasi-isometries. Recall that a quasi-isometry is a map $f\colon X \to Y$ between metric spaces for which there exist $K\geq0$ and $\lambda \geq 1$ such that:
\begin{itemize}
	\item[(i)] $f(X)$ is $K$-dense in $Y$;
	\item[(ii)] $\frac{1}{\lambda}d(x,x') - K \leq d(f(x),f(x')) \leq \lambda d(x,x') + K$ for all $x,x' \in X$.
\end{itemize}

\begin{prop}[\textup{\cite[Theorem 6.5]{BS11}}]
	\label{prop-isometry-quasisymmetry}
	Let $X,Y$ be two almost geodesic, Gromov-hyperbolic metric spaces and let $f\colon X \to Y$ be a quasi-isometry. Then $f$ induces a quasisymmetric homeomorphism $\partial f \colon \partial X \to \partial Y$.
\end{prop}
The statement means that for one (hence every) choice of metrics on $\mathcal{J}(\partial X)$ and $\mathcal{J}(\partial Y)$, the map $\partial f$ is a quasisymmetric homeomorphism. In this case we write $\mathcal{J}(\partial X) \cong \mathcal{J}(\partial Y)$ as in Proposition \ref{prop-convex-hull-boundary}.

\subsection{The proof of Theorem \ref{theorem-main-introduction}}
\label{sec-proof-B}
We recall the definition of the class $\mathcal{M}(\delta,D)$ appearing in Theorem \ref{theorem-main-introduction}.
Let $X$ be a proper, geodesic, $\delta$-hyperbolic metric space.
Every isometry of $X$ acts naturally on $\partial X$ and the resulting map on $X\cup \partial X$ is a homeomorphism.
A group of isometries $\Gamma$ of $X$ is said discrete if it is discrete in the compact-open topology.
The {\em limit set} $\Lambda(\Gamma)$ of a discrete group of isometries $\Gamma$ is the set of accumulation points of the orbit $\Gamma x$ on $\partial X$, where $x$ is any point of $X$. The group $\Gamma$ is called {\em elementary} if $\# \Lambda(\Gamma) \leq 2$. The set $\Lambda(\Gamma)$ is closed and $\Gamma$-invariant so it is its quasiconvex hull. A discrete group of isometries $\Gamma$ is {\em quasiconvex-cocompact} if its action on QC-Hull$(\Lambda(\Gamma))$ is cocompact, i.e. if there exists $D\geq 0$ such that for all $x,y\in \text{QC-Hull}(\Lambda(\Gamma))$ it holds $d(gx,y)\leq D$ for some $g\in \Gamma$. The smallest $D$ satisfying this property is called the {\em codiameter} of $\Gamma$.\vspace{2mm}

\noindent\emph{Given two real numbers $\delta \geq 0$ and $D>0$ we define $\mathcal{M}(\delta,D)$ to be the class of triples $(X,x,\Gamma)$, where $X$ is a proper, geodesic, $\delta$-hyperbolic metric space, $\Gamma$ is a discrete, non-elementary, torsion-free, quasiconvex-cocompact group of isometries with codiameter $\leq D$ and $x\in \textup{QC-Hull}(\Lambda(\Gamma))$.} 
\vspace{2mm}

\noindent Let $\Gamma$ be a finitely generated. Given a finite generating set $\Sigma$ of $\Gamma$ one can construct the Cayley graph $\text{Cay}(\Gamma,\Sigma)$ of $\Gamma$ relative to $\Sigma$. Any two Cayley graphs, made with respect to different generating sets, are quasi-isometric. $\Gamma$ is said to be Gromov-hyperbolic if one (and hence all) of its Cayley graphs is Gromov-hyperbolic. If it is the case the Gromov boundaries of every two Cayley graphs are quasisymmetric equivalent, by Proposition \ref{prop-isometry-quasisymmetry}. We denote the corresponding quasisymmetric gauge by $\mathcal{J}(\partial \Gamma)$.
A straightforward modification of the classical proof of the Svarc-Milnor lemma (along the same lines of \cite[Lemma 5.1]{Cav21ter}) says that every Cayley graph of $\Gamma$ is quasi-isometric to $\text{QC-Hull}(\Lambda(\Gamma))$, if $(X,x,\Gamma) \in \mathcal{M}(\delta, D)$. So both these spaces are Gromov-hyperbolic and almost geodesic. By Proposition \ref{prop-isometry-quasisymmetry} the gauges $\mathcal{J}(\partial \Gamma)$ and $\mathcal{J}(\partial \textup{QC-Hull}(\Lambda(\Gamma))) \cong \mathcal{J}(\Lambda(\Gamma))$ are quasisymmetric equivalent. The last equality is Proposition \ref{prop-convex-hull-boundary}. We can already prove the last part of Theorem \ref{theorem-main-introduction}.

\begin{prop}
	\label{prop-continuity}
	If $(X_n, x_n, \Gamma_n) \underset{\textup{eq-pGH}}{\longrightarrow} (X_\infty, x_\infty, \Gamma_\infty)$, with $(X_n, x_n, \Gamma_n) \in \mathcal{M}(\delta,D)$, then $\lim_{n\to + \infty}\textup{CD}(\Lambda(\Gamma_n)) = \textup{CD}(\Lambda(\Gamma_\infty))$.
\end{prop}
\noindent Here $(X_n, x_n, \Gamma_n) \underset{\textup{eq-pGH}}{\longrightarrow} (X_\infty, x_\infty, \Gamma_\infty)$ means that the triples $(X_n, x_n, \Gamma_n)$ converge in the equivariant pointed Gromov-Hausdorff sense to the triple $(X_\infty, x_\infty, \Gamma_\infty)$. We will not recall the definition, details can be found for instance in \cite{Cav21ter}. If all the triples $(X_n, x_n, \Gamma_n)$ belong to $\mathcal{M}(\delta,D)$ then also $(X_\infty, x_\infty, \Gamma_\infty) \in \mathcal{M}(\delta,D)$, by \cite[Theorem A]{Cav21ter}. In particular it is meaningful to talk about $\Lambda(\Gamma_\infty)$. Recall that the conformal dimensions are the conformal dimensions of the quasisymmetic gauges $\mathcal{J}(\Lambda(\Gamma_n))$, $n\in \mathbb{N} \cup \lbrace \infty \rbrace$. 

\begin{proof}
	By \cite[Theorem A]{Cav21ter} the triple $(X_\infty,x_\infty,\Gamma_\infty)$ belongs to $\mathcal{M}(\delta,D)$. Moreover \cite[Corollary 7.7]{Cav21ter} implies that $\Gamma_n$ is isomorphic to $\Gamma_\infty$ for $n$ big enough. Using Proposition \ref{prop-isometry-quasisymmetry} we conclude that $\mathcal{J}(\partial\Gamma_n) \cong \mathcal{J}(\partial\Gamma_\infty)$. The discussion above says that $\mathcal{J}(\Lambda(\Gamma_n)) \cong \mathcal{J}(\Lambda(\Gamma_\infty))$. Therefore, by definition, $\textup{CD}(\Lambda(\Gamma_n)) = \textup{CD}(\Lambda(\Gamma_\infty))$ for $n$ big enough.
\end{proof}

The next step is to show that, under the assumptions of Theorem \ref{theorem-main-introduction}, the spaces $\Lambda(\Gamma_n)$ are uniformly perfect and uniformly quasi-selfsimilar, when equipped with suitable visual metrics. Let $(X,x,\Gamma)\in \mathcal{M}(\delta,D)$. We always consider a standard visual metric $D_x$ centered at $x$ and with parameter $a_\delta = \frac{1}{4\delta \log_2e}$. All the estimates will be done with respect to this metric $D_x$.

\noindent The following three results are essentially known, see for instance \cite{Kle06}. We provide quantified version of them. The critical exponent of $\Gamma$ is $h_\Gamma = \lim_{T\to + \infty} \frac{1}{T} \log \# \Gamma x \cap \overline{B}(x,T).$

\begin{prop}
	\label{prop-Ahlfors-regularity}
	Let $\delta,D,H\geq 0$. There exists $A = A(\delta,D,H) >0$ such that for all $(X,x,\Gamma) \in \mathcal{M}(\delta,D)$ with $h_\Gamma \leq H$ the limit set $\Lambda(\Gamma)$ is $(A,\frac{h_\Gamma}{a_\delta})$-Ahlfors regular.
\end{prop}
\begin{proof}
	It follows from \cite[Theorem 6.1 and Lemma 4.9]{Cav21ter}.
\end{proof}
\begin{cor}
	\label{cor-hyp-uniform-perfect}
	Let $\delta,D,H\geq 0$. There exists $a_0 = a_0(\delta,D,H)$ such that for all $(X,x,\Gamma) \in \mathcal{M}(\delta,D)$ with $h_\Gamma \leq H$ the limit set $\Lambda(\Gamma)$ is $a_0$-uniformly perfect.
\end{cor}
\begin{proof}
	\cite[Proposition 5.2]{Cav21ter} says that $h_\Gamma \geq \frac{\log 2}{99\delta + 10D}$. The conclusion follows by Proposition \ref{prop-Ahlfors-regularity} and Lemma \ref{lemma-Ahlfors-perfect}.
\end{proof}

\begin{prop}
	\label{prop-hyp-qss}
	Let $\delta, D \geq 0$. There are $L_0 = L_0(\delta, D)$ and $\rho_0 = \rho_0(\delta, D)$ such that for all $(X,x,\Gamma)\in \mathcal{M}(\delta,D)$ the set $\Lambda(\Gamma)$ is $(L_0,\rho_0)$-q.s.s.
\end{prop}

Before the proof of this last property we need a bit of preparation.
\begin{lemma}[\textup{\cite[Lemma 4.2]{Cav21ter}}]
	\label{product-rays}
	Let $X$ be a proper, geodesic, $\delta$-hyperbolic metric space, $z,z'\in \partial X$ and $x\in X$. 
	\begin{itemize}
		\item[(i)] If $(z,z')_{x} \geq T$ then $d(\xi_{x,z}(T - \delta),\xi_{x,z'}(T - \delta)) \leq 4\delta$.
		\item[(ii)] If $d(\xi_{x,z}(T),\xi_{x,z'}(T)) < 2b$ then $(z,z')_{x} > T - b$, for all $b>0$.
	\end{itemize}
\end{lemma}

\begin{lemma}[\textup{\cite[Lemma 4.4]{Cav21ter}}]
	\label{parallel-geodesics}
	Let $X$ be a proper, geodesic, $\delta$-hyperbolic metric space. Then every two geodesic rays $\xi, \xi'$ with same endpoints at infinity are at distance at most $8\delta$, more precisely there exist $t_1,t_2\geq 0$ such that $t_1+t_2=d(\xi(0),\xi'(0))$ and  $d(\xi(t + t_1),\xi'(t+t_2)) \leq 8\delta$ for all $t\geq 0$.
\end{lemma}

Recall that on $\partial X$ we always consider a visual metrics of parameter $a_\delta$.

\begin{cor}
	\label{cor-product-rays}
	Let $(X,x,\Gamma)\in \mathcal{M}(\delta,D)$ and let $z,z'\in \partial X$. Let $\rho > 0$ and $R$ be such that $e^{-a_\delta R} =\rho$. If $D_x(z,z') \leq \rho$ then $d(\xi_{x,z}(R), \xi_{x,z'}(R)) \leq 14\delta$.
\end{cor}
\begin{proof}
	With this choice of $a_\delta$ we have $\frac{1}{2}e^{-a_\delta(z,z')_{x}}\leq D_x(z,z')\leq e^{-a_\delta(z,z')_{x}}$.
	So, if $D_x(z,z') \leq \rho$ then $(z,z')_x \geq R - \frac{\log(2)}{a_\delta} = R - 4\delta$. By Lemma \ref{product-rays} we get $d(\xi_{x,z}(R - 5\delta), \xi_{x,z'}(R-5\delta)) \leq 4\delta$, so by triangle inequality $d(\xi_{x,z}(R), \xi_{x,z'}(R)) \leq 14\delta$.
\end{proof}

We can finally give the

\begin{proof}[Proof of Proposition \ref{prop-hyp-qss}]
	We claim that $\rho_0=e^{-a_\delta \cdot D}$ works. We fix $0<\rho\leq \rho_0$ and we call $R\geq 0$ the real number such that $\rho = e^{-a_\delta \cdot R}$.
	Let $z\in \text{QC-Hull}(\Lambda(\Gamma))$ and $\xi_{x,z}$ be a geodesic ray joining $x$ to $z$. By Lemma \ref{approximation-ray-line} there is a point $y\in \text{QC-Hull}(\Lambda(\Gamma))$ such that $d(\xi_{x,z}(R),y) \leq 14\delta$. Moreover by definition of quasiconvex-cocompactness there exists $g\in \Gamma$ such that $d(y, gx) \leq D$, so $d(\xi_{x,z}(R), gx) \leq 14\delta + D$. Observe that $d(x,gx)\leq R + 14\delta + D$. We call $\Phi$ the map induced by $g$ on $\Lambda(\Gamma)$, which is well defined since $\Lambda(\Gamma)$ is $\Gamma$-invariant. We claim it satisfies the properties required by Definition \ref{defin-qss}.\\
	Let $w,w'$ be two points of $B(z,\rho) \cap \Lambda(\Gamma)$, so $D_x(w,z)$, $D_x(w',z)< \rho$. Let $T \geq 0$ be the real number such that $e^{-a_\delta \cdot T} = D_x(w,w')$. Since $D_x(w,w') < 2\rho$ we get $e^{-a_\delta \cdot T} < 2e^{-a_\delta \cdot R}$ and, by definition of $a_\delta$, $T>R-4\delta$. We apply three times Corollary \ref{cor-product-rays} to obtain
	\begin{equation*}
		\begin{aligned}
			&d(\xi_{x,z}(R),\xi_{x,w}(R)) \leq 14\delta,\\ &d(\xi_{x,z}(R),\xi_{x,w'}(R)) \leq 14\delta,\\ &d(\xi_{x,w}(T),\xi_{x,w'}(T)) \leq 14\delta.
		\end{aligned}
	\end{equation*}
	By triangle inequality we have
	\begin{equation*}
		\begin{aligned}
			\vert T-R \vert -D-28\delta \leq 
			d(x,g\xi_{x,w}(T)) 
			\leq \vert T-R \vert +D+28\delta.
		\end{aligned}
	\end{equation*}
	Similar estimates hold for $d(x,g\xi_{x,w'}(T))$. We want to estimate $d(\xi_{x,gw}(\vert T-R \vert), g\xi_{x,w}(T))$.
	The two rays $\xi_{x,gw}$ and $g\xi_{x,w}$ define the same point $gw$ of $\partial X$ and $d(x,gx)\leq R+14\delta+D$, so by Lemma \ref{parallel-geodesics} there exist $t_1,t_2\geq 0$ with $t_1+t_2 \leq R+14\delta+D$ such that $d(\xi_{x,gw}(t+t_1), g\xi_{x,w}(t+t_2)) \leq 8\delta$ for all $t\geq 0$. We apply this property to $t=T-t_2 + 18\delta + D$, which is non-negative since $T\ge R-4\delta$, finding 
	$$d(\xi_{x,gw}(T-t_2+t_1 + 18\delta + D), g\xi_{x,w}(T + 18\delta + D)) \leq 8\delta.$$
	By this inequality and the estimates on $d(x,g\xi_{x,w}(T))$ we get
	$$\vert T - R \vert - 2D - 54\delta \leq d(x,\xi_{x,gw}(T-t_2+t_1+18\delta+D)) \leq \vert T - R \vert + 2D + 54\delta,$$
	so 
	$$\vert T - R \vert - 2D - 54\delta \leq T-t_2+t_1+18\delta+D \leq \vert T - R \vert + 2D + 54\delta.$$
	Therefore by triangle inequality
	$$d(\xi_{x,gw}(\vert T-R \vert), g\xi_{x,w}(T)) \leq 80\delta + 3D.$$
	Analogously we get $d(\xi_{x,gw'}(\vert T-R \vert), g\xi_{x,w'}(T)) \leq 80\delta + 3D.$
	Combining these two estimates we conclude $d(\xi_{x,gw}(\vert T-R \vert), \xi_{x,gw'}(\vert T-R \vert))\leq 160\delta + 6D$. By Lemma \ref{product-rays} we have $(gw,gw')_x > \vert T - R \vert - 80 \delta - 3D$, so
	\begin{equation*}
		\begin{aligned}
			D_x(gw,gw') &\leq e^{-a_\delta(\vert T - R \vert - 80\delta - 3D)} \\
			&\leq e^{a_\delta(80\delta + 4D)}\cdot \frac{e^{-a_\delta D}}{e^{-a_\delta R}}\cdot e^{-a_\delta T} \\
			&= e^{a_\delta(80\delta + 4D)}\cdot \frac{\rho_0}{\rho}\cdot D_x(w,w')   ,
		\end{aligned}
	\end{equation*}
	where we used the definition of standard visual metric, $\vert T - R \vert \leq T-R$ and $\rho_0 = e^{-a_\delta D}$. We prove now the other inequality. We have $D_x(w,w') = e^{a_\delta T} \leq e^{a_\delta (w,w')_x}$, so $(w,w')_x \leq T$. We set $b=47\delta + D$ and $T' = T + b$. By Lemma \ref{product-rays}.(ii) we know that $d(\xi_{x,w}(T'), \xi_{x,w'}(T')) \geq 2b$. 
	We can argue in the same way as before with $t=T'-t_2$ finding
	$$d(\xi_{x,gw}(T'-R), g\xi_{x,w}(T')) \leq 44\delta + D,$$
	and the analogous estimate for $w'$. Therefore
	$$d(\xi_{x,gw}(T'-R), \xi_{x,gw'}(T'-R)) \geq 2b - 88\delta - 2D > 4\delta.$$
	By Lemma \ref{product-rays}.(i) we have $(gw,gw')_x < T'-R + \delta = T - R + 48\delta + D$.
	Therefore
	$$D_x(gw,gw') \geq \frac{1}{2}e^{-a_\delta(T - R + 48\delta + D)} = \frac{1}{2}e^{-48\cdot a_\delta\cdot\delta}\cdot \frac{\rho_0}{\rho}\cdot D_x(w,w').$$
	Thus $\Phi$ is $L_0$-biLipschitz with $L_0 = L_0(\delta,D) = \max\lbrace e^{a_\delta(80\delta + 4D)}, 2e^{48\cdot a_\delta\cdot\delta}\rbrace$ from $(B(z,\rho)\cap \Lambda(\Gamma), \frac{\rho_0}{\rho}\cdot D_x) \to \Lambda(\Gamma)$.
	We need to show that $\Phi(B(z,\rho)\cap \Lambda(\Gamma))\supseteq B(\Phi(z),\frac{\rho_0}{L_0})\cap \Lambda(\Gamma)$. The map $\Phi$ is a well defined self-homeomorphism of $\Lambda(\Gamma)$, so every $w\in B(\Phi(z),\frac{\rho_0}{L_0}) \cap \Lambda(\Gamma)$ is of the form $\Phi(w')$ for some $w'\in \Lambda(\Gamma)$. We know that $D_x(\Phi(w'), \Phi(z)) \leq \frac{\rho_0}{L_0}$, then $D_x(w',z) \leq L_0\cdot \frac{\rho}{\rho_0}\cdot D_x(\Phi(w'), \Phi(z)) \leq \rho$, i.e. $w'\in B(z,\rho) \cap \Lambda(\Gamma)$. This concludes the proof.
\end{proof}

\begin{cor}
	\label{cor-ultimo}
	Let $(X_n,x_n,\Gamma_n) \underset{\textup{eq-pGH}}{\longrightarrow} (X_\infty,x_\infty,\Gamma_\infty)$, with $(X_n,x_n,\Gamma_n) \in \mathcal{M}(\delta,D)$. Then the spaces $\Lambda(\Gamma_n)$ are uniformly q.s.s. and uniformly perfect.
\end{cor}
\begin{proof}
	By Proposition \ref{prop-hyp-qss} all the spaces $\Lambda(\Gamma_n)$ are compact and $(L_0,\rho_0)$-q.s.s. By
	\cite[Corollary 5.9]{Cav21ter} there exists $H\geq 0$ such that $h_{\Gamma_n} \leq H$ for all $n\in \mathbb{N}$. Then $\Lambda(\Gamma_n)$ is $a_0$-uniformly perfect for the same $0<a_0<1$, by Corollary \ref{cor-hyp-uniform-perfect}.
\end{proof}

The last step we need is the following.

\begin{prop}
	\label{prop-convergence-visual-distances}
	If $(X_n,x_n,\Gamma_n) \underset{\textup{eq-pGH}}{\longrightarrow} (X_\infty,x_\infty,\Gamma_\infty)$, with $(X_n,x_n,\Gamma_n) \in \mathcal{M}(\delta,D)$, then there exist visual metrics $D_n \in \mathcal{J}(\Lambda(\Gamma_n))$ for $n\in \mathbb{N} \cup \lbrace \infty \rbrace$ such that $(\Lambda(\Gamma_n), D_n)\underset{\textup{GH}}{\longrightarrow}(\Lambda(\Gamma_\infty), D_\infty)$, up to a subsequence.
\end{prop}

\begin{proof}
	We fix a non-principal ultrafilter $\omega$.
	We denote by $(X_\omega, x_\omega, \Gamma_\omega)$ the ultralimit triple of the sequence $(X_n,x_n, \Gamma_n)$ (cp. \cite{Cav21ter}): it is equivariantly isometric to $(X_\infty,x_\infty,\Gamma_\infty)$ by \cite[Proposition 3.13]{Cav21ter}. We equip each $\Lambda(\Gamma_n)$ with a standard visual metric $D_n$ of center $x_n$ and parameter $a_\delta$.	
	Every point of the space $\omega$-$\lim(\Lambda(\Gamma_n), D_{n})$ is an equivalence class of sequences $(z_n)$ with $z_n\in \Lambda(\Gamma_n)$. Associated to this sequence there is a sequence of geodesic rays $\xi_{x_n,z_n}$ of $X_n$. It is classical (cp \cite{CavS20}, Lemma A.7) that this sequence of geodesic rays define a limit geodesic ray $\xi_{x_\omega, z_\omega}$, with $z_\omega \in \partial X_\omega$. The map $\Psi \colon \omega$-$\lim(\Lambda(\Gamma_n), D_{n}) \to \partial X_\omega$ defined by $\Psi((z_n)) = z_\omega$ is a well defined homeomorphism by \cite[Proposition 5.11]{Cav21ter}. Moreover the proof of  \cite[Theorem A.(i)]{Cav21ter} shows that the image of $\Psi$ is exactly $\Lambda(\Gamma_\omega)$. We denote by $D_\omega$ the distance induced by $\Psi$ on $\Lambda(\Gamma_\omega)$. By definition the two spaces $\omega$-$\lim(\Lambda(\Gamma_n), D_n)$ and $(\Lambda(\Gamma_\omega), D_\omega)$ are isometric. Proposition \ref{prop-GH-ultralimit} implies that, up to a subsequence, $(\Lambda(\Gamma_n), D_n) \underset{\textup{GH}}{\longrightarrow} (\Lambda(\Gamma_\omega), D_\omega)$. We claim that $D_\omega$ is a visual metric on $\Lambda(\Gamma_\omega)$. This would imply, since $(X_\infty, x_\infty, \Gamma_\infty)$ and $(X_\omega, x_\omega, \Gamma_\omega)$ are equivariantly isometric, that $(\Lambda(\Gamma_n), D_n) \underset{\textup{GH}}{\longrightarrow} (\Lambda(\Gamma_\infty), D_\infty)$ for a visual metric $D_\infty$ on $\Lambda(\Gamma_\infty)$.\\
	Let us prove the claim.
	We take two points $\Psi(z),\Psi(z') \in \Lambda(\Gamma_\omega)$, with $z=\omega$-$\lim z_n$, $z' = \omega$-$\lim z_n' \in \omega$-$\lim(\Lambda(\Gamma_n), D_{n})$. By definition $D(\Psi(z),\Psi(z')) = \omega$-$\lim D_{n}(z_n,z_n') = e^{-a_\delta R}$ for some $R\geq 0$. From \eqref{eq-new-visual} we have $\omega$-$\lim (z_n,z_n')_{x_n} \geq R - 4\delta$, then $\omega$-$\lim d(\xi_{x_n,z_n}(R-5\delta), \xi_{x_n,z_n'}(R-5\delta)) \leq 4\delta$ by Lemma \ref{product-rays}.(i). This implies $d(\xi_{x_\omega,z_\omega}(R-5\delta), \xi_{x_\omega,z_\omega'}(R-5\delta)) \leq 4\delta$ and so $(\Psi(z), \Psi(z'))_{x_\omega} > R-8\delta$ by Lemma \ref{product-rays}.(ii). This means 
	$$D_\omega(\Psi(z),\Psi(z')) = e^{-a_\delta R} \geq e^{-a_\delta \cdot 8\delta}\cdot e^{-a_\delta(\Psi(z), \Psi(z'))_{x_\omega}} = \frac{1}{4}\cdot e^{-a_\delta(\Psi(z), \Psi(z'))_{x_\omega}}.$$
	Analogously, with the same notation as above, we have $\omega$-$\lim (z_n,z_n')_{x_n} \leq R$ and so $\omega$-$\lim d(\xi_{x_n,z_n}(R+3\delta), \xi_{x_n,z_n'}(R+3\delta)) \geq 6\delta$ by Lemma \ref{product-rays}.(ii). By definition $d(\xi_{x_\omega,z_\omega}(R+3\delta), \xi_{x_\omega, z_\omega'}(R+3\delta)) \geq 6\delta$, so $(\Phi(z), \Phi(z'))_{x_\omega} \leq R+4\delta$ by Lemma \ref{product-rays}.(i). This means
	$$D_\omega(\Psi(z),\Psi(z')) = e^{-a_\delta R} \leq e^{a_\delta \cdot 4\delta}\cdot e^{-a_\delta(\Psi(z), \Psi(z'))_{x_\omega}} = 2\cdot e^{-a_\delta(\Psi(z), \Psi(z'))_{x_\omega}}.$$
	This shows that $D_\omega$ is a visual metric on $\Lambda(\Gamma_\omega)$ and concludes the proof.
\end{proof}
\vspace{2mm}
\noindent Theorem \ref{theorem-main-introduction} follows  from Corollary \ref{cor-ultimo}, Proposition \ref{prop-convergence-visual-distances} and Proposition \ref{prop-continuity}.\linebreak

\bibliographystyle{alpha}
\bibliography{Continuity_of_conformal_dimension}

\begin{thebibliography}{CDP90}

\bibitem[BH13]{BH09}
M.~Bridson and A.~Haefliger.
\newblock {\em Metric spaces of non-positive curvature}, volume 319.
\newblock Springer Science \& Business Media, 2013.

\bibitem[BK13]{BK13}
M.~Bourdon and B.~Kleiner.
\newblock Combinatorial modulus, the combinatorial loewner property, and
  coxeter groups.
\newblock {\em Groups, Geometry, and Dynamics}, 7(1):39--107, 2013.

\bibitem[BM17]{BM17}
M.~Bonk and D.~Meyer.
\newblock {\em Expanding thurston maps}, volume 225.
\newblock American Mathematical Soc., 2017.

\bibitem[BS11]{BS11}
M.~Bonk and O.~Schramm.
\newblock Embeddings of gromov hyperbolic spaces.
\newblock {\em Selected Works of Oded Schramm}, pages 243--284, 2011.

\bibitem[Cav21]{Cav21ter}
N.~Cavallucci.
\newblock Continuity of critical exponent of quasiconvex-cocompact groups under
  gromov-hausdorff convergence.
\newblock {\em arXiv preprint arXiv:2105.11764}, 2021.

\bibitem[CDP90]{CDP90}
M.~Coornaert, T.~Delzant, and A.~Papadopoulos.
\newblock {\em G{\'e}om{\'e}trie et th{\'e}orie des groupes: les groupes
  hyperboliques de Gromov}.
\newblock Lecture notes in mathematics. Springer-Verlag, 1990.

\bibitem[CS20]{CavS20bis}
N.~Cavallucci and A.~Sambusetti.
\newblock Discrete groups of packed, non-positively curved, gromov hyperbolic
  metric spaces.
\newblock {\em arXiv preprint arXiv:2102.09829}, 2020.

\bibitem[CS21]{CavS20}
N.~Cavallucci and A.~Sambusetti.
\newblock Packing and doubling in metric spaces with curvature bounded above.
\newblock {\em Mathematische Zeitschrift}, pages 1--46, 2021.

\bibitem[DK18]{DK18}
C.~Dru{\c{t}}u and M.~Kapovich.
\newblock {\em Geometric group theory}, volume~63.
\newblock American Mathematical Soc., 2018.

\bibitem[DS97]{DS97}
G.~David and S.~Semmes.
\newblock {\em Fractured fractals and broken dreams: self-similar geometry
  through metric and measure}, volume~7.
\newblock Oxford University Press, 1997.

\bibitem[Ha{\"\i}07]{Hai07}
P.~Ha{\"\i}ssinsky.
\newblock G{\'e}om{\'e}trie quasiconforme, analyse au bord des espaces
  m{\'e}triques hyperboliques et rigidit{\'e}s, d’apr{\`e}s mostow, pansu,
  bourdon, pajot, bonk, kleiner.
\newblock {\em S{\'e}minaire Bourbaki}, 60:08, 2007.

\bibitem[Hei01]{Hei01}
J.~Heinonen.
\newblock {\em Lectures on analysis on metric spaces}.
\newblock Springer Science \& Business Media, 2001.

\bibitem[Jan17]{Jan17}
D.~Jansen.
\newblock Notes on pointed gromov-hausdorff convergence.
\newblock {\em arXiv preprint arXiv:1703.09595}, 2017.

\bibitem[Kle06]{Kle06}
B.~Kleiner.
\newblock The asymptotic geometry of negatively curved spaces: uniformization,
  geometrization and rigidity.
\newblock 2:743--768, 2006.

\bibitem[Kwa20]{Kwa20}
J.~Kwapisz.
\newblock Conformal dimension via p-resistance: Sierpinski carpet.
\newblock 45(1), 2020.

\bibitem[MT10]{MT10}
J.~Mackay and J.~Tyson.
\newblock {\em Conformal dimension: theory and application}, volume~54.
\newblock American Mathematical Soc., 2010.

\bibitem[Pau96]{Pau96}
F.~Paulin.
\newblock Un groupe hyperbolique est déterminé par son bord.
\newblock {\em Journal of the London Mathematical Society}, 54(1):50--74, 1996.

\bibitem[Pia11]{Pia11}
M.C. Piaggio.
\newblock {\em Jauge conforme des espaces m{\'e}triques compacts}.
\newblock PhD thesis, Universit{\'e} de Provence-Aix-Marseille I, 2011.

\bibitem[Pia12]{Pia12}
M.C. Piaggio.
\newblock Conformal dimension and combinatorial modulus of compact metric
  spaces.
\newblock {\em Comptes Rendus Mathematique}, 350(3-4):141--145, 2012.

\bibitem[Sul82]{Sul82}
D.~Sullivan.
\newblock Seminar on hyperbolic geometry and conformal dynamical systems.
\newblock {\em preprint IHES}, 1982.

\bibitem[TV80]{TV80}
P.~Tukia and J.~V{\"a}is{\"a}l{\"a}.
\newblock Quasisymmetric embeddings of metric spaces.
\newblock {\em Ann. Acad. Sci. Fenn. Ser. A I Math.}, 5:97--114, 1980.

\bibitem[Tys00]{Tys00}
J.~Tyson.
\newblock Sets of minimal hausdorff dimension for quasiconformal maps.
\newblock {\em Proceedings of the American Mathematical Society},
  128(11):3361--3367, 2000.

\end{thebibliography}

\end{document}